\documentclass[11pt]{article}
\usepackage{amsmath,amssymb,amsthm}
\usepackage{comment,enumerate}
\usepackage{thm-restate}

\usepackage[margin=1in]{geometry}
\usepackage{hyperref}
\usepackage{tikz}
\usepackage{graphicx}
\usepackage{url}
\usepackage[mathlines]{lineno}
\usepackage{multirow}
\usepackage{dsfont} 
\usepackage{color}
\definecolor{red}{rgb}{1,0,0}
\def\red{\color{red}}
\definecolor{blue}{rgb}{0,0,1}

\definecolor{green}{rgb}{0,.6,0}

\usepackage{tikz}
\usetikzlibrary{positioning}
\tikzset{
every picture/.style=thick,
bluenode/.style={circle, draw=black, fill=blue!40, very thick, minimum size=6mm},
whitenode/.style={circle, draw=black, fill=black!5, very thick, minimum size=2mm},
squarednode/.style={rectangle, draw=red!60, fill=red!5, very thick, minimum size=5mm},
every loop/.style={min distance=8mm} 
}

\newtheorem{thm}{Theorem}[section]
\newtheorem{cor}[thm]{Corollary}
\newtheorem{lem}[thm]{Lemma}
\newtheorem{prop}[thm]{Proposition}

\theoremstyle{definition}
\newtheorem{rem}[thm]{Remark}

\theoremstyle{definition}
\newtheorem{defn}[thm]{Definition}

\theoremstyle{definition}
\newtheorem{ex}[thm]{Example}

\def\mtx#1{\begin{pmatrix} #1 \end{pmatrix}}

\DeclareMathOperator{\supp}{supp}

\DeclareMathOperator{\TS}{TS_S}
\DeclareMathOperator{\TSS}{TS_S}
\DeclareMathOperator{\TSM}{TS_M}
\DeclareMathOperator{\TSA}{TS_A}
\newcommand{\Ec}{\overline{E}}

\DeclareMathOperator{\tr}{tr}
\DeclareMathOperator{\Span}{span}
\DeclareMathOperator{\mult}{mult}
\newcommand{\verA}{\Psi_A}
\newcommand{\verM}{\Psi_M}
\newcommand{\verS}{\Psi_S}
\newcommand{\Gc}{\overline{G}}

\newcommand{\R}{\mathbb{R}}

\newcommand{\Rn}{\R^{n}}

\newcommand{\cV}{\mathcal{V}}
\newcommand{\EOLam}{\mei_{\mathcal{O}\cup\Lambda'}}

\newcommand{\bx}{{\bf x}}
\newcommand{\bv}{{\bf v}}
\newcommand{\bw}{{\bf w}}
\newcommand{\bz}{{\bf z}}

\newcommand{\bb}{{\bf b}}
\newcommand{\ba}{{\bf a}}

\newcommand{\bff}{{\bf f}}
\newcommand{\by}{{\bf y}}

\newcommand{\be}{{\bf e}}
\newcommand{\bm}{{\bf m}}
\newcommand{\bzero}{{\bf 0}}
\newcommand{\oml}{{\bf m}}

\newcommand{\dcup}{\,\dot{\cup}\,}
\newcommand{\x}{\times}
\newcommand{\lam}{\lambda}
\newcommand{\noi}{\noindent}
\newcommand{\bit}{\begin{itemize}}
\newcommand{\eit}{\end{itemize}}
\newcommand{\ben}{\begin{enumerate}}
\newcommand{\een}{\end{enumerate}}
\newcommand{\beq}{\begin{equation}}
\newcommand{\eeq}{\end{equation}}
\newcommand{\bea}{\begin{eqnarray*}}
\newcommand{\eea}{\end{eqnarray*}}
\newcommand{\bean}{\begin{eqnarray}}
\newcommand{\eean}{\end{eqnarray}}
\newcommand{\bpf}{\begin{proof}}
\newcommand{\epf}{\end{proof}\ms}
\newcommand{\bmt}{\begin{bmatrix}}
\newcommand{\emt}{\end{bmatrix}}
\newcommand{\ms}{\medskip}

\newcommand{\beqa}{\begin{array}}
\newcommand{\eeqa}{\end{array}}

\newcommand{\mptn}{\mathcal{S}} 
\newcommand{\mrk}{\mathcal{R}} 
\newcommand{\mei}{\mathcal{E}} 
\newcommand{\mmu}{\mathcal{U}} 
\newcommand{\G}{\mathcal{G}}

\newcommand{\vect}{\operatorname{vec}}

\DeclareMathOperator{\Mplus}{M_{+}}

\DeclareMathOperator{\M}{M}

\def\spec{\operatorname{spec}}

\def\trans{^{\top}}

\begin{document}
\title{The inverse eigenvalue problem of a graph: \\
Multiplicities and minors}
\author{Wayne Barrett\thanks{Department of Mathematics, Brigham Young University, Provo UT 84602, USA (wb@mathematics.byu.edu, H.Tracy@gmail.com).}
\and Steve Butler\thanks{Department of Mathematics, Iowa State University, Ames, IA 50011, USA (butler@iastate.edu, hogben@iastate.edu, chlin@iastate.edu, myoung@iastate.edu)}
\and Shaun M. Fallat\thanks{Department of Mathematics and Statistics, University of Regina, Regina, Saskatchewan, S4S0A2, Canada (shaun.fallat@uregina.ca).}
 \and H. Tracy Hall\footnotemark[1]
\and  Leslie Hogben\footnotemark[2]\, \thanks{American Institute of Mathematics, 600 E. Brokaw Road, San Jose, CA 95112, USA (hogben@aimath.org).}
\and Jephian C.-H. Lin\footnotemark[2]
\and Bryan L.  Shader\thanks{Department of Mathematics, University of Wyoming,  Laramie, WY 82071, USA (bshader@uwyo.edu).}
\and Michael Young\footnotemark[2]  
}

 

\maketitle

\begin{abstract} The inverse eigenvalue problem of  a given graph $G$ is to determine all possible spectra of  real symmetric matrices whose off-diagonal entries are governed by
the adjacencies in $G$.  Barrett et al. introduced the Strong Spectral Property (SSP) and the Strong Multiplicity  Property (SMP)  in \cite{genSAP}.  In that paper it was shown that if a graph  has a matrix with the SSP (or the SMP) then a supergraph has a matrix with the same spectrum (or ordered multiplicity list) augmented with simple eigenvalues if necessary, that is, subgraph monotonicity. In this paper we extend this to a form of minor monotonicity, with restrictions on where the new eigenvalues appear.  These ideas are applied to solve the inverse eigenvalue problem for all graphs of order five, and to  characterize forbidden minors of graphs having at most one multiple eigenvalue.
\end{abstract}

\medskip


\section{Introduction}


 Inverse eigenvalue problems appear in various contexts throughout mathematics and engineering, and refer to determining all possible lists of eigenvalues (spectra)  for matrices fitting some description. Graphs often describe relationships in a physical setting, such as control of a system, and the eigenvalues of associated matrices govern the behavior of the system.   The {\em inverse eigenvalue problem of a graph (IEPG)} refers to determining the possible spectra of real symmetric matrices whose pattern of nonzero off-diagonal entries is described by the edges of a given graph.  The IEPG is motivated by inverse problems arising in 
the theory of vibrations.  The study of the vibrations of a string leads to  classical inverse 
Sturm-Liouville problems, and its generalizations continue to be an active area of research 
(see \cite{G} and references therein). The IEPG where $G$ is a path corresponds to a discretization 
of the  inverse Sturm-Liouville problem for the string, and leads to the classical study of the inverse
eigenvalue problem for Jacobi matrices (that is, irreducible, tridiagonal matrices) 
that was resolved by the sequence of papers by Downing and Householder \cite{DH},  and Hochstadt \cite{H1, H2}.  Thus, as noted in \cite{G}, the IEPG can be viewed as  the 
inverse problem for a vibrating system with prescribed structure given by $G$.

The IEPG was originally approached through   the study of  ordered multiplicity lists for eigenvalues.  It was thought by many researchers in the field that at least for a tree $T$, determining the ordered multiplicity lists of $T$ would suffice to determine the spectra of matrices described by $T$. When it was shown in \cite{BF04} that this was not the case, the focus of much of the research in the area shifted to the narrower question of maximum eigenvalue multiplicity, or equivalently maximum nullity or minimum rank of matrices described by the graph.  While the maximum multiplicity has been determined for many  families of graphs, including all trees, in general it remains an open question and active area of research (see \cite{FH, HLA2} for  extensive bibliographies).  More recently, there has been progress on the related question of determining the minimum number of distinct eigenvalues of matrices  described by a given graph \cite{AACFMN13, genSAP}.  

Maximum nullity, minimum number of distinct eigenvalues, and ordered multiplicity lists all provide information that can in some cases be used to solve the inverse eigenvalue problem for a specific graph or family of graphs, but the question of the structures or properties that are necessary  to allow this to be done more generally is open. Recently new tools,  the Strong Spectral Property (SSP) and the Strong Multiplicity Property (SMP), were developed \cite{genSAP}; these  offer hope for progress on what was previously seen as an intractable problem.

In Section \ref{sord5} we use the SSP to solve the IEPG for graphs of order 5 (order 4 was solved in \cite{BNSY14}, but the SSP provides a shorter proof).  
Every  graph of order 5 that allows an SMP matrix for an ordered multiplicity list also allows an SSP matrix for the same ordered multiplicity list. This naturally raises the question of whether there exists a graph that has a matrix $A\in\mptn(G)$ with the SMP, but does not allow the SSP for any matrix having  the ordered multiplicity list of $A$.  In Section \ref{sSMPnotSSP} we exhibit such a matrix and establish its properties.  (Of course it is much easier to construct a matrix that has the SMP but not the SSP; see, for example,  \cite[Remark~22]{genSAP}.)

Another  main result of this paper, which we use to establish characterizations of multiple eigenvalues by forbidden minors, is the Minor Monotonicity Theorem (Theorem \ref{minormon}).  This theorem says that if $G$ is a minor of $H$ and $A\in\mptn(G)$ has the SSP (or the SMP), then there is a matrix $B\in\mptn(H)$ with the SSP (or the SMP) and $\spec(A)\subseteq \spec(B)$ (or the ordered multiplicity list of $B$ can be obtained from that of $A$ by adding 1s).  Additional eigenvalues are added as simple eigenvalues and can always be added at the ends of the spectrum of $A$; in the case of a vertex deletion they can be added arbitrarily, but for a minor obtained by contraction the new eigenvalues may be restricted to being sufficiently far from the spectrum of $A$.

Minor monotonicity enables forbidden minor characterizations, and as an illustration of this  in Section \ref{smultevals} we use Theorem \ref{minormon} to determine the forbidden minors for a graph to have at most one multiple eigenvalue and to not have two consecutive multiple eigenvalues.   

In Section \ref{sMLL} we establish additional tools for constructing matrices with prescribed spectra and multiplicity lists. The foundation  is the Matrix Liberation Lemma (Lemma \ref{mtxliblem}), which describes how the three strong properties indicate the ability of a matrix
to effect any sufficiently small perturbation of its nonzero entries with complete freedom
while preserving its rank (SAP), its exact spectrum (SSP), or its ordered list of eigenvalue multiplicities (SMP).
One consequence of this is the Augmentation Lemma (Lemma \ref{shaun}), which gives  conditions under which  an eigenvalue $\lambda$ of an SSP matrix $A\in \mptn(G-v)$ guarantees the existence of a matrix $B\in\mptn(G)$ with $\spec(B)=\spec(A)\cup\{\lambda\}$ (i.e., $\mult_B(\lambda)=\mult_A(\lambda)+1$ and the other eigenvalues and multiplicities are unchanged from $A$ to $B$).


 In the next section we introduce the necessary terminology, including definitions of the SSP and the SMP.  Sections \ref{sminormon} and \ref{sMLL} 
 are rather technical and do not use any results from Section \ref{sord5}, \ref{sSMPnotSSP}, and \ref{smultevals}, so we 
 defer the proofs of the results therein to the later part of the paper.


\section{Terminology, notation, and background}
\label{sterm}

 All matrices are real and symmetric; $O$ and $I$ denote  zero and identity matrices of appropriate size, respectively. If the distinct eigenvalues of $A$ are $\lambda_1< \lambda_2< \cdots< \lambda_q$ and the multiplicities 
of these eigenvalues are $m_1, m_2, \ldots, m_q$ respectively, then the {\it ordered multiplicity list of $A$} 
is $\oml(A)=(m_1, m_2, \ldots, m_q)$.  The {\em spectral radius} of $A$ is $\rho(A)=\max \{|\lam|: \lam\in\spec(A)\}$.  
For an $n\x n$ matrix $M$ and $\alpha,\beta \subseteq \{1,2,\ldots,n\}$,
the submatrix of $M$ lying in rows indexed by
$\alpha$ and columns indexed by $\beta$ is denoted
by $M[\alpha,\beta]$; in the case that $\beta=\{1,2,\ldots,n\}$ this can be denoted by $M[\alpha, :]$, and similarly for $\alpha=\{1,2,\ldots,n\}$.

A symmetric matrix $A$ has the {\it Strong Arnold Property} (or $A$ has the SAP for short)
if the only symmetric matrix $X$ satisfying $A\circ X=O$, $I\circ X=O$ and $AX=O$ is $X=O$.  
An $n \times n$ symmetric matrix $A$ satisfies the {\it Strong Multiplicity Property} (or $A$ has the SMP)
provided the only symmetric matrix $X$ satisfying $A\circ X=O$, $I\circ X=O$, $[A,X]=O$, and $\tr(A^iX)=0$ for $i=2,\ldots, n-1$ is $X=O$ \cite[Definition~18 and Remark~19]{genSAP}. 
A symmetric matrix $A$ has the {\it Strong Spectral Property} (or $A$ has the SSP)
if the only symmetric matrix $X$ satisfying $A\circ X=O$, $I\circ X=O$ and $[A,X]=O$ is $X=O$ \cite[Definition~8]{genSAP}.  Clearly the SSP implies the SMP, and the SMP implies $A+\lambda I$ has the SAP for every real number $\lambda$.

The {\em graph} $\G(A)$ of a symmetric $n\x n$ matrix $A$ is the (simple, undirected, finite) graph with vertices $\{1,\dots,n\}$ and edge $ij$ such that $i\ne j$ and $a_{ij} \neq 0$.   For a  graph $G=(V,E)$ with  vertex set $V=\{1,\dots,n\}$ and edge set $E$, 
the {\em set of symmetric matrices described by $G$}, $\mathcal{S}(G)$,  is the set of all real symmetric $n \times n$ matrices $A=[a_{ij}]$  such that  $\G(A)=G$.  The IEPG for $G$ asks for the determination of all possible spectra of matrices in $\mathcal{S}(G)$.
The  {\em maximum multiplicity}  of $G$ is
$\M(G) = \max\{\operatorname{mult}_A (\lambda) :  A 
\in \mathcal{S}(G),\;\lambda\in \operatorname{spec}(A)\}$,  and the {\em minimum rank} of $G$ is 
$\operatorname{mr}(G) = \min\{\operatorname{rank } A :  A \in \mathcal{S}(G)\}$.  It is easily seen that $\M(G)=\max\{\operatorname{null } A :  A 
\in \mathcal{S}(G)\}$, so $\M(G)$ is also called the {\em maximum nullity} of $G$, and $\operatorname{mr}(G)+\M(G)=|G|$, where $|G|$ is the number of vertices of $G$. The number of distinct eigenvalues of $A$ is denoted by $q(A)$, and   $q(G) = \min\{q(A)\; \mbox{:}\; A \in \mptn(G)\}$.

If $v$ is a vertex of a graph $G$, the {\em neighborhood} of $n$ is the set of vertices adjacent to $v$, and is denoted by $N_G(v)$. 
The {\em closed neighborhood} of $v$ is $N_G[v]:=N_G(v)\cup \{v\}$.   The {\em complement} $\overline{G}$ of $G$ is the graph with the same vertex set as $G$ and edges exactly where $G$ does not have edges.  A {\em generalized  star} is a tree with at most one vertex of degree three or more.  A {\em unicyclic graph} is a graph with one cycle.

The well known Parter-Wiener Theorem for trees plays a fundamental role in the study of the IEPG, and we state it here.

\begin{thm}[Parter-Wiener Theorem]\label{PWthm} {\rm \cite{JLDS2, P, W84}}  Let $T$ be a tree, $A\in\mptn(T)$ and $\mult_A(\lam)\ge 2$. Then there exists a vertex $v$ such that $\mult_{A(v)}(\lam)=\mult_A(\lam)+1$ and $\lam$ is an eigenvalue  of the principal submatrices of $A$ corresponding to at least three components of $T-v$.
\end{thm}

There are several (well known) consequences of Theorem \ref{PWthm} and interlacing  (see, for example, \cite
{JLS03}).

\begin{lem}\label{firstlast}   The first and last eigenvalue of a tree must be simple.
\end{lem}

\begin{lem}\label{collision} If $T$ is a  generalized star and $A\in\mptn(T)$, then $A$ cannot have consecutive multiple eigenvalues. \end{lem}

\section{Inverse eigenvalue problem for graphs of order at most five}\label{sord5}

In this section we solve the IEP for all graphs of order at most five, both in general and with the stipulation that there is a matrix realizing each possible spectrum that has the SSP (or equivalently, the SMP).  The solution to the IEP for graphs of order at most three is well known and straightforward, and the IEP for graphs of order 4 was solved in \cite{BNSY14}.  In Section \ref{ss4} we briefly re-derive the solution for order 4 by using the SSP to shorten the arguments.  We then solve the IEP for graphs of order 5 in Section \ref{ss5}. 

An ordered multiplicity list $\oml=(m_1,\dots,m_k)$ is {\em spectrally arbitrary} for graph $G$ if any set of $k$ real numbers $\lambda_1<\dots < \lambda_k$ can be realized as the spectrum of $A\in\mptn(G)$ with $\mult_A(\lambda_i)=m_i$  
(and $\oml$   is spectrally arbitrary for $G$ with the SSP has the obvious interpretation).  Figure \ref{fig:diagram} summarizes the solution for connected graphs of order at most 5, both with the SSP and without the SSP.  Note that some graph names are abbreviated: Banner (Bnr), Butterfly (Bfly), Diamond (Dmnd), Full House (FHs), House (Hs).  We show that for a connected graph of order  at most 4, every ordered multiplicity list that is attainable by the graph is spectrally arbitrary with the SSP.  For order 5 every ordered multiplicity list that is attainable (respectively, attainable with the SSP) is spectrally arbitrary (spectrally arbitrary with the SSP), but there are connected graphs and ordered multiplicity lists that can be realized only without the SSP.  In all cases disconnected graphs have some ordered multiplicity lists that can be realized only without the SSP.
While it is not true that all trees are spectrally arbitrary \cite{BF04}, we show that it is true for all graphs of order at most five.

  \begin{figure}[!h]
\begin{center}
\scalebox{.85}{\includegraphics{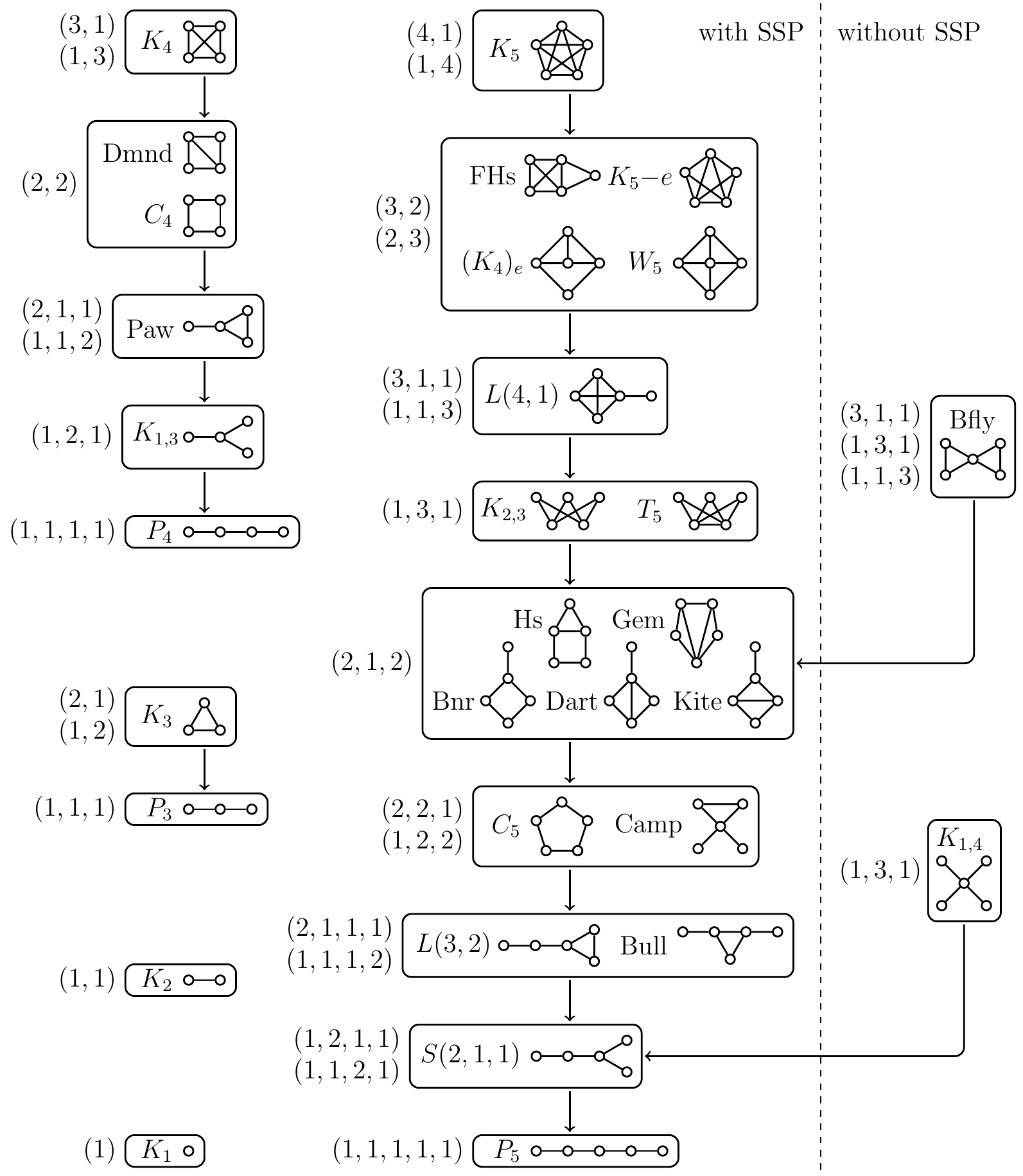}}
\caption{The connected graphs of order at most 5 with their ordered multiplicity lists.  If a box is joined to another box by a line  then the graphs in the upper box can realize every ordered multiplicity list of the graphs in the lower box (including other boxes below connected with lines to lower boxes). Every ordered multiplicity list is spectrally arbitrary for the graphs that attain it.}\label{fig:diagram}\vspace{-15pt}
\end{center}
\end{figure}

\begin{rem}\label{discon}  Spectra for disconnected graphs can be determined from those of connected graphs by allowing  any permissible assignment of disjoint spectra  to the connected components when SSP is required \cite[Theorem~34]{genSAP}, and when SSP is not required  any permissible assignment of spectra  to the connected components.  Thus the solutions to the inverse eigenvalue problem with and without SSP are different.
\end{rem}


\begin{rem}\label{2evals}Given a matrix $A$ that has exactly two distinct eigenvalues $\lam_1<\lam_2$ with multiplicities $m_1$ and $m_2$, and any possible pair of  real numbers $\mu_1<\mu_2$, the matrix \[B=\frac{\mu_2-\mu_1}{\lam_2-\lam_1}(A-\lam_1I)+\mu_1I\] has eigenvalues $\mu_1$ with multiplicity $m_1$ and $\mu_2$ with multiplicity $m_2$.  This technique is referred to as {\em scale and shift}.  By negation of the matrix, the order of multiplicities can be reversed.  Negation and scale and shift all preserve the SSP, meaning that if $A$ has the SSP, then so does $\alpha A+\beta I$ for any nonzero real number $\alpha$ and any $\beta\in\R$.  Thus for an ordered multiplicity list with only two multiplicities, exhibiting one matrix (respectively, one matrix with the SSP) with this ordered multiplicity list suffices to show the graph is spectrally arbitrary (respectively, spectrally arbitrary with the SSP) for this ordered multiply list.
\end{rem}

For every graph of order $n$, any set of $n$ distinct real numbers is attained by a matrix with the SSP \cite[Remark~15]{genSAP}.  Thus  $(1,1,\dots,1)$ is spectrally arbitrary with SSP for every graph.  Only distinct eigenvalues are possible  for a path.  
This covers all connected graphs of order at most 3 except $K_3$.  Any $A\in\mptn(K_3)$ has the SSP, and  $J_3$ has ordered multiplicity list $(2,1)$. The list (3) cannot be attained by a  connected graph.

If $G$ is a subgraph of $H$ on the same vertex set, then any spectrum attained by $G$ with SSP  is also attained by $H$ with SSP \cite[Theorem~10]{genSAP}.  Thus it is useful to identify minimal graphs attaining a given ordered multiplicity list and show that these are spectrally arbitrary.  These minimal subgraphs need not be connected.

We state an additional result that will be used.

\begin{lem}\label{firstlast2} Suppose $G$ is a connected unicyclic graph with an odd cycle.  At least one of the first and last eigenvalues of $G$ must be simple.
\end{lem}
\bpf Let $A\in\mptn(G)$.  If the cycle product (product of the entries of $A$ on the cycle) is negative, then replace $A$ by $-A$; since the cycle is odd, the cycle product is now positive.  Let $G'$ be obtained from $G$ by deleting one edge $\{i,j\}$ of the  cycle, so $G'$ is a spanning tree of $G$.  Let $A'$ be defined from $A$ by replacing (the equal) entries $a_{ij}$ and $a_{ji}$ by 0, so $A'\in\mptn(G')$.  There exists a diagonal matrix $D$ with diagonal entries $\pm 1$  so that $DA'D=DA'D^{-1}$ is a nonnegative matrix \cite[Lemma 1.2]{DHHHW}. Since the cycle product of $A$ is unchanged by a diagonal similarity, the cycle product of $DAD$ is positive, implying $DAD$ is nonnegative.
Since $DAD$ is symmetric, it has 2-cycles and it has an odd cycle, so $DAD$ is primitive. The spectral radius of a primitive nonnegative matrix is  simple, so the last eigenvalue of $A$ is simple.
\epf
 \vspace{-10pt}


\subsection{Order 4}\label{ss4}

The next result describes the solution to the IEP for connected graphs of order 4; as noted in Remark \ref{discon},  the solution for orders 1, 2, and 3 suffice to solve the IEP for disconnected graphs of order 4.

\begin{prop}\label{prop4IEP} $\null$
\ben
\item Table \ref{OML4tab}
lists all the minimal subgraphs with respect to the SSP for each ordered multiplicity list  of order $4$. 
In each case the reversal of the given list has the same minimal subgraphs. 
\item The order $4$ part of Figure \ref{fig:diagram}
lists  all the possible ordered multiplicity lists for each  connected graph of order $4$ (both those ordered multiplicity lists next to the box with  the graph and all those next to graphs below and connected by a sequence of lines  in the order $4$ diagram).  Each ordered multiplicity list in the order $4$ part of  Figure \ref{fig:diagram} is spectrally arbitrary with the SSP.   \een \end{prop}\vspace{-10pt}

\begin{table}[h!]\caption{Minimal order 4 subgraphs for attainment  with SSP, by ordered multiplicity list (OML)}\vspace{-10pt}\label{OML4tab}
\begin{center} 
\noindent {\small \begin{tabular}{|c|cc|cc|cc|}
\hline 
OML & graph  & reason & graph & reason & graph & reason \\
\hline 
(3,1) & $K_4$ & $J_4$ &  &   && \\
\hline 
(2,2) & $C_4$ & $\mtx{0 & 1 & 0 & -1\\1 & 0 & 1 & 0\\ 0 & 1 & 0 & 1\\-1 & 0 & 1 &0 }$ &  &    &&\\
\hline 
(1,2,1) & $C_4$ & Corollary \ref{Cn211} 
 &{$K_3\dcup K_1$} &  {\normalsize $\left(\frac{\lam_3-\lam_2}3J_3+\lam_2I_3\right) \oplus [\lam_1
]$} & $K_{1,3}$ &  $\mtx{a & b & b & b \\
 b & 0 & 0 & 0 \\
 b & 0 & 0 & 0 \\
 b & 0 & 0 & 0}$   \\
\hline 
(2,1,1) & $C_4$ & Corollary \ref{Cn211}& {$K_3\dcup K_1$} & {\normalsize  $\left(\frac{\lam_2-\lam_1}3J_3+\lam_1I_3\right) \oplus [\lam_3]$} & & \\[1mm]
\hline 
(1,1,1,1) & $4K_1$ & \cite[Remark~15]{genSAP} &  & & &   \\
\hline 
\end{tabular}}\vspace{-15pt}
\end{center}
\end{table}

\bpf
In each case a matrix or reason  is listed in  Table \ref{OML4tab}.  It is straightforward to verify that  each matrix has the SSP regardless of the parameters, and to see that the matrix for  $K_{1,3}$ for (1,2,1) is spectrally arbitrary by choice of parameters (and scale and shift); see \cite{suppdocs} for verifications of the SSP. Thus each ordered multiplicity list in Table \ref{OML4tab} is spectrally arbitrary with the SSP for the listed graph. 

It is then straightforward to verify that every ordered multiplicity list in the order 4 part of Figure \ref{fig:diagram} is spectrally arbitrary with SSP for the graphs in  its box or above it, by use of a subgraph with the ordered multiplicity list that is spectrally arbitrary with SSP  (cf. Table \ref{OML4tab}).  

To complete the proof, we justify that  no other ordered multiplicity list is possible for each graph, which also justifies that all minimal subgraphs are listed in Table \ref{OML4tab}:  \vspace{-4pt}
\bit
\item For $K_4$, $\M(K_4)=3$. \vspace{-4pt}
\item For $C_4$ and Diamond, $\M(C_4)=2=\M({\sf Dmnd})$. \vspace{-4pt}
\item Paw has at least 3 distinct eigenvalues ($q($Paw$)=3$) by \cite[Theorem 3.2]{AACFMN13} since  there is a unique path of length 2 from the degree one vertex to either degree two vertex.   \vspace{-4pt}
\item For $K_{1,3}$, the first and last eigenvalues are simple (Lemma \ref{firstlast}) and $\M(K_{1,3})=2$.  \vspace{-4pt}
\item For $P_4$,  the maximum multiplicity is $\M(P_4)=1$.  \qedhere
\eit \epf


\subsection{Order 5}\label{ss5}

\begin{lem}\label{5matrixlemma} Let 
\[  M_{1}={\scriptsize
\mtx{-t^4+2 t^3-t^2-1 & 0 & -(t-1) t & (t-1)^2 t^2 & 0 \\
 0 & -t^4+2 t^3-t^2-1 & 0 & -(t-1) t^2 & -(t-1) t \\
 -(t-1) t & 0 & -t^2 \left(t^2-2 t+2\right) & 0 & -(t-1) t^2 \\
 (t-1)^2 t^2 & -(t-1) t^2 & 0 & -t^2 \left(t^2-2 t+2\right) & 0 \\
 0 & -(t-1) t & -(t-1) t^2 & 0 & -2 (t-1)^2 t^2 \\
}\!\!, \, 0<t<1,} \vspace{-3pt}\]
\[M_{2}= \mtx{-1 & 1 & -a & 0 & 0 \\
 1 & -1 & -a & 0 & 0 \\
 -a & -a & 2 a^2-2 & -a & -a \\
 0 & 0 & -a & 0 & 0 \\
 0 & 0 & -a & 0 & 0 }\!\!, \,a\ne 0,\qquad 
 M_{3}=\mtx{ 1 & -1 & 0 & 0 & -a \\
 -1 & 1 & 0 & 0 & -a \\
 0 & 0 & -a^2 & a^2 & a \\
 0 & 0 & a^2 & -a^2 & a \\
 -a & -a & a & a & 2-2 a^2 }\!\!, \,a\ne 0,\vspace{-3pt}\]
\[M_{4}= \mtx{a & 0 & b & b & b \\
 0 & -a c^2 & b c & b c & b c \\
 b & b c & 0 & 0 & 0 \\
 b & b c & 0 & 0 & 0 \\
 b & b c & 0 & 0 & 0}\!\!, \,b\ne 0,c\ne 0,\pm 1,\qquad M_{5}=\mtx{  a^2 & 0 & \sqrt{2} a & a & a \\
 0 & 1 & -\sqrt{2} & 1 & 1 \\
 \sqrt{2} a & -\sqrt{2} & 4 & 0 & 0 \\
 a & 1 & 0 & 2 & 2 \\
 a & 1 & 0 & 2 & 2 }\!\!, \,a\ge 1.\vspace{-3pt}\]
\ben[(1)]
\item $M_i, i=1,2,3,4,5$ has the SSP for any permitted parameters.
\item $\G(M_1)=C_5$, $\G(M_2)= {\sf Camp}$, $\G(M_3)={\sf Bfly}$, $\G(M_4)=K_{2,3}$, and $\G(M_5)=(K_4)_e$.
\item\label{C5} $\oml(M_{1})=(2,2,1)$, and $\lam_1<\lam_2<\lam_3$ can be realized  with $M_{1}$ by choosing $t$ and scale and shift.
\item\label{G4} $\oml(M_2)=(2,2,1)$, and $\lam_1<\lam_2<\lam_3$ can be realized  with $M_{2}$ by choosing $a$ and scale and shift.
\item\label{G11} $\oml(M_3)=(2,1,2)$, and $\lam_1<\lam_2<\lam_3$ can be realized  with $M_3$ by choosing $a$ and scale and shift.
\item\label{K23} $\oml(M_4)=(1,3,1)$, and $\lam_1<\lam_2<\lam_3$ can be realized  with $M_4$ by choosing $a,b,c$ and scale and shift.
\item\label{G17} $\oml(M_5)=(3,1,1)$ for $a> 1$, and $\lam_1<\lam_2<\lam_3$ can be realized  by $M_5$ by choosing $a>1$ and scale and shift.  $\oml(M_5)=(3,2)$ for $a= 1$, and $\lam_1<\lam_2$ can be realized  by $M_5$ with $a=1$ by  scale and shift.
\een
\end{lem}
\bpf  In each case it is clear that $M_i$ has the stated graph, and it is straightforward to verify the SSP for each matrix (see \cite{suppdocs}). 

For statements \eqref{G4}, \eqref{G11}, and \eqref{G17}, examination of the eigenvalues shows that the claimed ordered multiplicity list is realized, and any spectrum can be realized by appropriate choice of the parameter and scale and shift; here we list the spectra:\vspace{-6pt}
\bea \spec(M_2)=\{-2, -2, 0, 0, 2 a^2\}\\
\spec(M_3)=\{ -2 a^2, -2 a^2,0,2, 2\}\\
\spec(M_5)=\{ 0, 0, 0, 5, 4 + a^2\}\vspace{-8pt} \eea

Computations show that $\spec(M_1)=\{\lam,\lam,\mu,\mu,0\}$, with \vspace{-6pt}
\bea&\lam &= 
\frac{1}{2} \left(-3 t^4+6 t^3-4 t^2-1-(1-t) \sqrt{t^6-2 t^5+3 t^4+3 t^2+2 t+1}\right)\\
   < &\mu&= 
   \frac{1}{2} \left(-3 t^4+6 t^3-4 t^2-1+(1-t) \sqrt{t^6-2 t^5+3 t^4+3 t^2+2 t+1}\right)\\
   < &0, &\vspace{-10pt}\eea
with the last inequality following from $\lam<0$ and Lemma \ref{firstlast2}.  Note that $\lam$ and $\mu$ are continuous functions of $t$ with $\lim_{t\to0}\frac\mu{\lam}=0$ and $\lim_{t\to1}\frac\mu{\lam}=1$.  So given $\alpha_1<\alpha_2<0$, we can choose $t$ so that $\frac\mu{\lam}=\frac{\alpha_2}{\alpha_1}$.  Then scaling $M_1$ achieves eigenvalues 
$\{\alpha_1,\alpha_1,\alpha_2,\alpha_2,0\}$, and finally a shift is made if needed.

For the matrix $M_4$ with ordered multiplicity list (1,3,1), the characteristic polynomial is \vspace{-3pt}
\[p_{M_4}(x)=x^3 \left(-a^2 c^2 - 3 b^2 (1 + c^2) + a (-1 + c^2) x + x^2\right).\vspace{-3pt}\]  Any  eigenvalues $\lam$ and $\mu$ of opposite sign can be realized by 
\[a=  \frac{-\lam-\mu }{c^2-1},\
b=
   \frac{\sqrt{-c^4 \lam \mu -c^2 \lam^2-c^2 \mu ^2-\lam \mu  }}
   {\sqrt{3} \sqrt{c^6-c^4-c^2+1}}.\] The only restrictions on the parameters are $b\ne 0, c\ne 0, c\ne \pm1$, $-c^4 \lam \mu -c^2 \lam^2-c^2 \mu ^2-\lam \mu> 0$, and $c^6-c^4-c^2+1> 0$. Since $-\lambda  \mu>0$, all the inequalities can be ensured by choosing $c$  sufficiently small.
\epf\vspace{-10pt}

\begin{thm}\label{thm5IEP} $\null$
\ben
\item Table \ref{OML5tab} lists all the minimal subgraphs with respect to the SSP for each ordered multiplicity list  of order $5$.  
In each case the reversal of the given list has the same minimal subgraphs.
\item 
The order 5 part of Figure \ref{fig:diagram}
lists  all the possible ordered multiplicity lists for each  connected graph of order 5 (both those ordered multiplicity lists next to the box with  the graph and all those next to graphs below and connected by a sequence of lines  in the order 5 diagram).  Each ordered multiplicity list in Figure \ref{fig:diagram} is spectrally arbitrary.  Those to the left of the dashed vertical  line can  realized with the SSP, whereas those to the right cannot be realized with the SMP (and thus not with the SSP).  
\een  \end{thm}\vspace{-10pt}

\begin{table}[h!]\caption{Minimal order 5 subgraphs for attainment  with SSP, by ordered multiplicity  list (OML)}\vspace{-5pt}\label{OML5tab}
\begin{center} 
\noindent {\small \begin{tabular}{|c|cc|cc|cc|cc|}
\hline 
OML  & graph  & reason & graph & reason & graph & reason& graph & reason \\
\hline 
(4,1) & $K_5$ & $J_5$ &  &  & & & &  \\
\hline 
(3,2) & $(K_4)_e$ & $M_5, a=1$&  & & & & &     \\[0.3mm]
\hline 
(3,1,1) &$K_4\dcup K_1$& & $(K_4)_e$ & $M_5, a>1$   &  & &  &\\[0.3mm]
\hline 
(1,3,1) &$K_4\dcup K_1$& & $K_{2,3}$ & $M_4$  &&  & &\\[0.3mm]
\hline 
(2,1,2) &$C_4\dcup K_1$ &&  Butterfly & $M_3$ & &  & & \\[0.3mm]
\hline 
(2,2,1) &$C_4\dcup K_1$ &&  $C_5$ & $M_1$ &  Campstool  & $M_2$  & & \\[0.3mm]
\hline
(2,1,1,1) & $C_4\dcup K_1$ & 
& $C_5$  & Corollary \ref{Cn211} & $K_3\dcup 2K_1$ & & & \\
\hline 
(1,2,1,1) & $C_4\dcup K_1$ & 
& $C_5$  & Corollary \ref{Cn211} & $K_3\dcup 2K_1$ &  & $K_{1,3}\dcup K_1$ &  \\
\hline 
(1,1,1,1,1) & $5K_1$ & &  
& &  & &   &  \\
\hline \end{tabular}}\vspace{-15pt}
\end{center}
\end{table}



\bpf For each connected graph   in  Table \ref{OML5tab}, a matrix or reason  is listed.   That the matrices listed  have the given ordered multiplicity list  and can realize any such spectra was established in Lemma \ref{5matrixlemma}, and it is straightforward to verify that each of  the listed matrices has the SSP  (see \cite{suppdocs}).  For the disconnected graphs, the result follows from Proposition \ref{prop4IEP} and the Block Diagonal Theorem \cite[Theorem~34]{genSAP}.

The information in Table \ref{OML5tab} can be used to justify the following statement: Every ordered multiplicity list to the left of the dashed vertical line in the order 5 part of Figure \ref{fig:diagram}  is spectrally arbitrary with the SSP for the graphs in  its box or above it.  In some cases  a subgraph with the desired ordered multiplicity list that is spectrally arbitrary with the SSP  is used for the justification.  

To complete the proof, we show that for each graph no ordered multiplicity list other than those described in Figure \ref{fig:diagram} 
 is possible,  which also justifies that all minimal subgraphs are listed in Table \ref{OML5tab}, and discuss why the non-SSP ordered multiplicity lists cannot be realized with SMP. 
In each case below, the statements about  $\M$ (maximum multiplicity = maximum nullity), $\Mplus$ (maximum positive  semidefinite nullity) and $\xi$ (maximum nullity with SAP)  are well known (see, for example, \cite{AIM}, \cite{BFH3}, and \cite{HLA2}).  \vspace{-5pt}
\bit
\item
For $P_5$,  $\M(P_5)=1$. \vspace{-5pt}
\item
For $G=S(2,1,1)$, the first and last eigenvalues are simple since $G$ is a tree (Lemma \ref{firstlast}), and $\M(G)=2$. 
\vspace{-5pt}
\item
For $K_{1,4}$: The first and last eigenvalues are simple,  so the only possible ordered multiplicity list that has not already been shown to be realized with the SSP is   $(1,3,1)$. The adjacency matrix 
realizes $(1,3,1)$,   and it is known that any star is spectrally arbitrary for every ordered multiplicity list it attains \cite{BF04}.
    Since $\xi(K_{1,4})=2$ and the SMP implies the SAP, $(1,3,1)$ cannot be realized  with the SMP.  
\vspace{-5pt}
\item
For $G=L(3,2)$  or Bull,  $\M(G)=2$, $q(G)=4$, and one of the first and last eigenvalues is simple by Lemma \ref{firstlast2}. 
\vspace{-5pt}
\item
For $G=C_5$ or Campstool (Camp),  $\M(G)=2$ and one of the first and last eigenvalues is simple by Lemma \ref{firstlast2}. 
\vspace{-5pt}
\item
For $G$ one of the graphs Banner (Bnr), House (Hs), Dart, Gem, or Kite,  $\M(G)=2$. 
\vspace{-5pt}
\item
For $G=$ Butterfly (Bfly):  $\M(G)=3$ and $\xi(G)=2$, so no  ordered multiplicity list containing a 3 can be realized  with the SMP.  Without the SMP, it is known that (1,3,1) and (3,1,1) are spectrally arbitrary for Butterfly \cite[Section 5.2]{Kempton}.  \vspace{-5pt}
\item
For $G=K_{2,3}$ or $T_5$,  $\M(G)=3$ and $\Mplus(G)=2$; the latter eliminates $(3,1,1)$ and $(3,2)$ as possible ordered multiplicity lists.   
\vspace{-5pt}
\item
For $L(4,1)$, $\M(L(4,1))=3$ and $q(L(4,1))=3$ by unique shortest path of length 2. 
\vspace{-5pt}
\item
For $G$ one of Full House (FHs), $K_5-e$, $(K_4)_e$, or $W_5$, $\M(G)=3$.  \vspace{-5pt}
\item
For $K_5$, $\M(K_5)=4$.\qedhere
\eit\epf


\section{A graph and ordered multiplicity list that allow the SMP but not the SSP}\label{sSMPnotSSP}

In this section we exhibit a graph $G$ and ordered multiplicity list $\oml$ such if $A\in\mptn(G)$ and $\oml(A)=\oml$, then $A$ does not have the SSP, yet there exists $B\in\mptn(G)$ with $\oml(B)=\oml$ that has the SMP.   To establish that  for a given graph and ordered multiplicity list, no matrix can  have the SSP, we apply the next result.  

\begin{thm}{\rm \cite[Corollary~29(b)]{genSAP}} 
 \label{edgebdSSP}  
Suppose $G$ is a graph,  $A\in\mptn(G)$,  $\oml(A)=(m_1,\dots,m_q)$, and $A$ has the SSP.  Then   $|E(G)| \geq \sum_{j=1}^q {m_i \choose 2}$.
\end{thm}

\begin{prop}\label{prop:SMPnotSSP}  Let $H$ be the graph shown in Figure \ref{fig:SMPnotSSP}.  If $A\in\mptn(H)$ and  $\oml(A)=(3,5,4)$, then $A$ does not have the SSP.  The matrix 

\[B=\left(
{\scriptsize\begin{array}{rrrrrrrrrrrr}
 1 & -1 & 0 & 0 & 0 & -1 & \sqrt{2} & 0 & 0 & 0 & 0 & 0 \\
 -1 & 1 & -1 & 0 & 0 & 0 & 0 & \sqrt{2} & 0 & 0 & 0 & 0 \\
 0 & -1 & 1 & -1 & 0 & 0 & 0 & 0 & \sqrt{2} & 0 & 0 & 0 \\
 0 & 0 & -1 & 1 & -1 & 0 & \sqrt{2} & 0 & 0 & 0 & 0 & 0 \\
 0 & 0 & 0 & -1 & 1 & -1 & 0 & \sqrt{2} & 0 & 0 & 0 & 0 \\
 -1 & 0 & 0 & 0 & -1 & 1 & 0 & 0 & \sqrt{2} & 0 & 0 & 0 \\
 \sqrt{2} & 0 & 0 & \sqrt{2} & 0 & 0 & -2 & 0 & 0 & 2 & 0 & 0
   \\
 0 & \sqrt{2} & 0 & 0 & \sqrt{2} & 0 & 0 & -2 & 0 & 0 & 2 & 0
   \\
 0 & 0 & \sqrt{2} & 0 & 0 & \sqrt{2} & 0 & 0 & -2 & 0 & 0 & 2
   \\
 0 & 0 & 0 & 0 & 0 & 0 & 2 & 0 & 0 & 0 & 1 & 1 \\
 0 & 0 & 0 & 0 & 0 & 0 & 0 & 2 & 0 & 1 & 0 & 1 \\
 0 & 0 & 0 & 0 & 0 & 0 & 0 & 0 & 2 & 1 & 1 & 0 
\end{array}}
\right)\in\mptn(H)\]
has  $\spec(B)=\{-4,-4,-4,0,0,0,0,0,3,3,3,3\}, \oml(B)=(3,5,4)$, and $B$ has the SMP.
\end{prop}
\begin{figure}[!h]
\begin{center}
\scalebox{.35}{\includegraphics{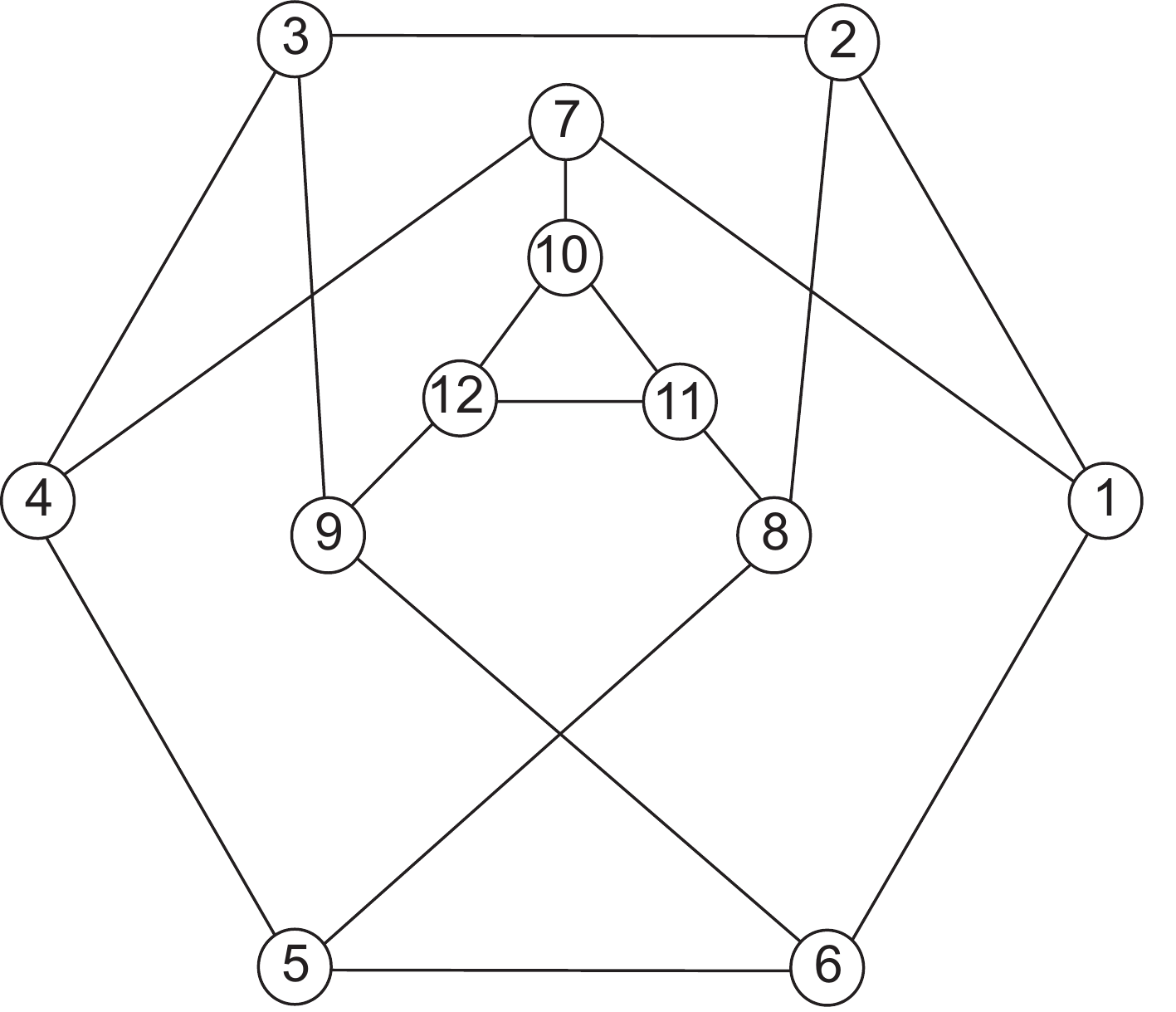}}
\caption{The graph $H$, which allows ordered multiplicity list (3,5,4) with the SMP but not with the SSP}\label{fig:SMPnotSSP}\vspace{-10pt}
\end{center}
\end{figure}%

\bpf    Consider a matrix $A\in \mptn(H)$ with $\oml(A)=(3,5,4)$.  Since ${3 \choose 2}+{5 \choose 2}+{4 \choose 2}=19>18=|E(H)|$,  $A$ does not have the SSP by Theorem \ref{edgebdSSP}.  It is straightforward to verify computationally that $B$ does have the SMP  (see \cite{suppdocs}) and has the stated spectrum.
\epf\vspace{-10pt}


\section{Minimal minors for multiple eigenvalues}\label{smultevals}
An eigenvalue is {\em multiple} if it is not simple, i.e., if it has multiplicity at least two.  In this section we determine the forbidden minors for a graph to have at most one multiple eigenvalue in a matrix with the SSP or the SMP, and characterize connected graphs that do not have two consecutive multiple eigenvalues.


\subsection{Minimal minors having at least two multiple eigenvalues}

\begin{thm} Let $G$ be a graph.  
\ben[(1)]
\item\label{thm611}
If $G$ is a connected graph and none of the eleven graphs shown in Figure \ref{fig:11minminor} is a minor of $G$, then any matrix $A\in\mptn(G)$ has at most  one multiple eigenvalue.
\item \label{thm612}
There exists a matrix $A \in \mptn(G)$ with SSP (or with the SMP) having more than one multiple eigenvalue if and only if one of the eleven graphs shown in Figure \ref{fig:11minminor} is a minor of $G$.
\een
\end{thm}

\begin{figure}[!h]
\begin{center}
\scalebox{.4}{\includegraphics{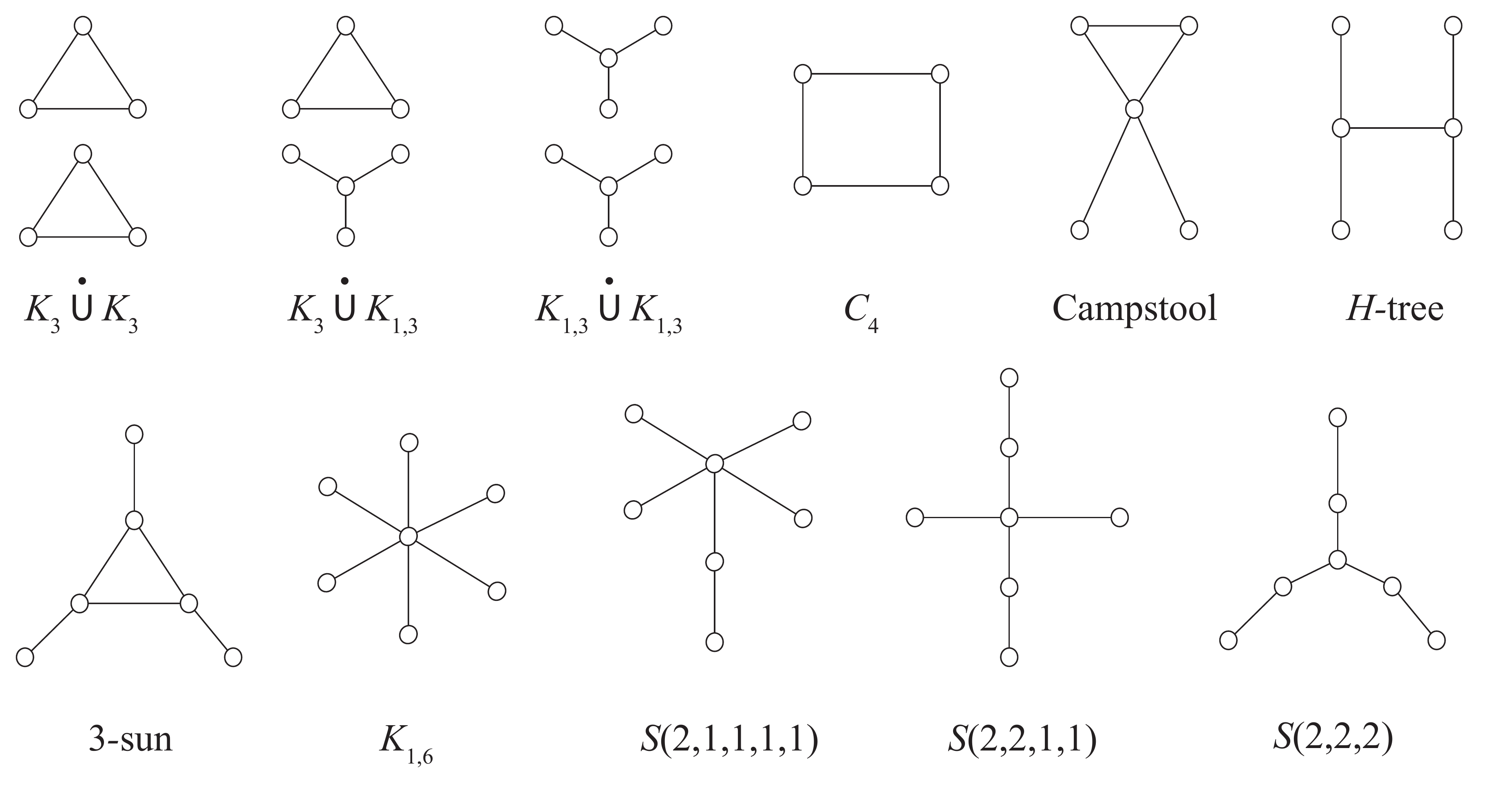}}
\caption{The eleven minimal minors for two multiple eigenvalues\label{fig:11minminor}}\vspace{-10pt}
\end{center}
\end{figure}

\bpf 

\noi\eqref{thm611} Suppose  that $G$ is a connected graph that has none of these eleven graphs as a minor.  We show that  $A \in \mptn(G)$  
can have at most one multiple eigenvalue.  The statement is immediate if $|G|\le 3$ so assume $|G|\ge 4$.  \smallskip

\underline{Case A:  $G$ contains a cycle.}  Since $C_4$ is not a minor of $G$, the only cycles in $G$ are triangles.  Moreover, there can be no pair of disjoint triangles in $G$ since  $K_3 \dcup K_3$ is not a minor of $G$, there can be no pair of  triangles intersecting in exactly a single vertex since Campstool is not a minor of $G$, and there can be no pair of triangles intersecting in exactly a single edge since $C_4$ is not a minor of $G$.  Thus there is a single 3-cycle $C$ in $G$.   Since the $3$-sun is not a minor of $G$, one of the vertices  of $C$ has degree $2$, and since  Campstool is not a minor of $G$, the degrees of the other two vertices of $C$ are at most $3$.  Since $|G|\ge 4$, at least one of the  vertices of $C$ has degree 3.  
Then all vertices not on $C$ must have degree at most $2$ or else the $H$-tree would be a minor of $G$.  It follows that $G$ is a path with an additional edge joining two vertices on the path at distance $2$.  By \cite[Proposition~50 and Theorem~51]{genSAP}, $q(G) = n-1$ and $A$ has at most a single multiple eigenvalue.  
\smallskip

\underline{Case B:  $G$ is a tree.}  Since the $H$-tree is not a minor of $G$, $G$ has at most one vertex $v$ of degree greater than $2$.   If $A$ has two distinct  eigenvalues of multiplicity at least $2$, then by the Parter-Wiener Theorem (Theorem \ref{PWthm}), $A[G-v]$ has two eigenvalues of multiplicity at least  $3$ and each of these must occur in at least 3 different components of $G-v$.  We show this is impossible.

The degree of $v$ is at most 5 since $K_{1,6}$ is not a minor of $G$.  If the degree of $v$ is $5$, then $G$ is  $K_{1,5}$ since $S(2, 1, 1, 1, 1)$ is not a minor of $G$, but this does not permit 3 components each for 2 eigenvalues.
Suppose the degree of $v$ is $4$.  Since $S(2,2,1,1)$ is not a minor of $G$, all but one of the neighbors of $v$ is a pendent vertex, which also does not permit 3 components each for 2 eigenvalues.
Suppose the degree of $v$ is 3.  Since $S(2,2,2)$ is not a minor, at least one neighbor of $v$ is pendent, but this does not permit 3 components each with 2 eigenvalues.
The only remaining case is that $G$ is a path, and then all eigenvalues of $A$ are simple.\smallskip

  \begin{table}[h!]
\begin{center} 
\caption{Matrices with two double eigenvalues and the SSP \label{tab11minminor}}\vspace{5pt}
\noindent {\small \renewcommand{\arraystretch}{1.3}\begin{tabular}{|c|c|c|c|}
\hline 
graph & matrix/reason & spectrum & OML  \\
\hline 
 $K_3\dcup K_3$ & Remark \ref{discon} &  & (2,2,1,1)  \\
\hline
$K_3 \dcup K_{1,3}$ & Remark \ref{discon} &  & $  (1, 2, 2, 1, 1)$\\
\hline
$K_{1,3} \dcup K_{1,3}$ & Remark \ref{discon} &
  & $  (1, 1, 2, 2, 1, 1)$\\
\hline 
 $C_4$ &Table \ref{OML4tab} & $\left\{-\sqrt{2},-\sqrt{2},\sqrt{2},\sqrt{2}\right\}$ & $  (2, 2)$
  \\
\hline
Campstool & \cite[Proposition~44]{genSAP}  &
 $ \{-2, 0, 0, 2, 2\}$ &$  (1, 2, 2)$\\
\hline
$H$-tree & \cite[Proposition~44]{genSAP} &  $  \left\{\frac 1 2(1 - \sqrt{29}), 0, 0, 1, 1, \frac 1 2(1 + \sqrt{29})\right\}$ & $ (1, 2, 2, 1)$\\
\hline
3-sun & \cite[Proposition~44]{genSAP} &  $ \left\{ 0, 0, \frac 1 2(5-\sqrt{13}), 2, 2,\, \frac 1 2(5+\sqrt{13})\right\}$ & $ (2, 1, 2, 1)$\\
\hline
$K_{1,6}$
&
${\scriptsize\mtx{
0 & 1 & 1 & 1 & 2 & 2 & 2 \\
 1 & 1 & 0 & 0 & 0 & 0 & 0 \\
 1 & 0 & 1 & 0 & 0 & 0 & 0 \\
 1 & 0 & 0 & 1 & 0 & 0 & 0 \\
 2 & 0 & 0 & 0 & 0 & 0 & 0 \\
 2 & 0 & 0 & 0 & 0 & 0 & 0 \\
 2 & 0 & 0 & 0 & 0 & 0 & 0}}$ &  $\left\{\frac 12(-3-\sqrt{21}),\, 0, 0,\, \frac12(-3+\sqrt{21}), 1, 1, 4\right\}$& $( 1, 2, 1, 2, 1)$\\
\hline
$S(2, 1, 1, 1, 1)$
&
${\scriptsize\mtx{
0 & 1 & 0 & 3 & 2 & 1 & 1 \\
 1 & 1 & 1 & 0 & 0 & 0 & 0 \\
 0 & 1 & 1 & 0 & 0 & 0 & 0 \\
 3 & 0 & 0 & 2 & 0 & 0 & 0 \\
 2 & 0 & 0 & 0 & 2 & 0 & 0 \\
 1 & 0 & 0 & 0 & 0 & 0 & 0 \\
 1 & 0 & 0 & 0 & 0 & 0 & 0
}}$ &  $ \left\{ \frac 1 2(-3 - \sqrt{13}),\, 0, 0,\, \frac 1 2(\sqrt{13} -3), 2, 2 \right\}$ & $ (1, 2, 1, 2, 1)$\\
\hline
 S(2, 2, 1, 1)
&
${\scriptsize\mtx{
0 & 2 & 0 & 2 & 0 & 1 & \sqrt{2} \\
 2 & 1 & 1 & 0 & 0 & 0 & 0 \\
 0 & 1 & 1 & 0 & 0 & 0 & 0 \\
 2 & 0 & 0 & 1 & 1 & 0 & 0 \\
 0 & 0 & 0 & 1 & 1 & 0 & 0 \\
 1 & 0 & 0 & 0 & 0 & 2 & 0 \\
 \sqrt{2} & 0 & 0 & 0 & 0 & 0 & 0 
}}$ &  $ \{ -3, 0, 0, 1, 2, 2, 4\}$ & $ (1, 2, 1, 2, 1)$\\
\hline
 S(2,2,2)& \cite[Proposition~44]{genSAP} & $ \{-1.4605, 0, 0, .760877, 2, 2, 2.69963\}$ & $(1, 2, 1, 2, 1)$
\\
\hline\end{tabular}}
\end{center}\vspace{-10pt} 
\end{table}

\eqref{thm612} ($\Rightarrow$):  Suppose  that $G$ does not have any  of these eleven graphs as a minor.  The case in which $G$ is connected is covered by \eqref{thm611}, so assume $G$ is disconnected.   
Since $K_3 \dcup K_3$ is not a minor,  $G$ has at most one component with a cycle.  Let $A \in S(G)$ have the SMP. 

\underline{Case A:   $G$ has a component that contains a cycle.} Call this component  $G_1$.  Since $K_3 \dcup K_{1,3}$ is not a minor of $G$, all other components of $G$ are paths.  By \cite[Theorem~34]{genSAP}, the spectra associated with the components must be disjoint, so all multiple eigenvalues must be associated with the connected component $G_1$.  Then by \eqref{thm611}, all but one of the eigenvalues associated with $G_1$ are simple.  

\underline{Case B: $G$ is a forest.}  Since $K_{1,3} \dcup K_{1,3}$ is not a minor, all but possibly one component are paths; denote the non-path component by  $T_1$.  Since the spectra associated with the components must be disjoint,  all multiple eigenvalues must be associated with the component $T_1$, and by \eqref{thm611}, all but one of the eigenvalues associated with $T_1$ are simple.  
\smallskip

\eqref{thm612} ($\Leftarrow$): For each graph in Figure \ref{fig:11minminor}, Table \ref{tab11minminor} lists one of the following: i) A matrix  in $\mptn(G)$  (or a citation of a reference that contains such a matrix)  together with its eigenvalues and ordered multiplicity list; the matrix has two eigenvalues of multiplicity $2$ and  has the SSP.  ii) A reason that implies the graph has a matrix with two multiple eigenvalues and the SSP.  
 Then, by Theorem \ref{minormon}, any graph that has one of these eleven graphs as a minor must allow a matrix with the SSP that has at least two multiple eigenvalues.
\epf

\newpage


\subsection{Minimal minors having at least two consecutive multiple eigenvalues}

Let $\lambda_1<\lambda_2<\cdots <\lambda_k$ be the distinct eigenvalues of a matrix.  Eigenvalues of the form $\lambda_i, \lambda_{i+1},\dots, \lambda_{i+s}$ with $s\ge 1$ are referred to as  {\em consecutive}.  Theorem \ref{thm:twoconmulti} below characterizes, by forbidden minors, the graphs that do not allow two consecutive multiple eigenvalues.  To prove this theorem, we need some additional  results and notation.  If $\mathcal{F}$ is a family of graphs, we say that a graph $G$ {\em does not have an $\mathcal{F}$-minor} if no graph in $\mathcal{F}$ is a minor of $G$.

Examining Table \ref{tab11minminor}, some of the eleven graphs in Figure \ref{fig:11minminor} allow a matrix with the SSP and  two consecutive multiple eigenvalues, including 
\[\{K_3\dcup K_3, K_3\dcup K_{1,3}, K_{1,3}\dcup K_{1,3},C_4, \text{Campstool}, H\text{-tree}\}.\]
The remaining graphs are $3\text{-sun},  K_{1,6}, S(2,1,1,1,1),S(2,2,1,1), S(2,2,2)$; all but the 3-sun are generalized stars.  

A \emph{generalized $3$-sun} is obtained from the 3-sun by subdividing each edge incident with a vertex of degree one as many times as desired  (note no subdivision is acceptable, so the 3-sun is also a generalized 3-sun).  It is well known that a generalized star $G$ does not allow a matrix with two  consecutive multiple eigenvalues (Lemma \ref{collision}), and we show this is also a property of a  generalized $3$-sun.


\begin{lem}\label{3sunlem}
A generalized $3$-sun does not allow a matrix with two consecutive multiple eigenvalues.
\end{lem}
\begin{proof}
Suppose $A\in \mptn(G)$ has consecutive multiple eigenvalues.  Since $G$ is unicyclic with an odd cycle, we may assume all off diagonal entries have the same sign.  By shifting and scaling we may also assume the two consecutive multiple eigenvalues are $1$ and $-1$, so $A^2-I$ is a positive semidefinite matrix with nullity at least $4$.

Define the graphs $G^{(2)}$ and $G^2$ by $V(G^{(2)})=V(G^2)=V(G)$,  $E(G^{(2)})$ contains all pairs of vertices that have distance $2$ in $G$, and $E(G^2)=E(G)\dcup E(G^{(2)})$. Let $H$ be the graph corresponding to the matrix $A^2-I$.   Then {$E(G^{(2)})\subseteq E(H)\subseteq E(G^2)$}. 
 Now let $S$ be the three vertices on the center triangle of $G$.  Then $S$ is a positive semidefinite zero forcing set of $H$, meaning $M_+(H)\leq Z_+(H)\leq 3$, which is a contradiction.  For the definition of $Z_+(G)$ and its relation with $M_+(G)$, see \cite{smallparam} or \cite{HLA2}.
\end{proof}

\begin{thm}\label{thm:twoconmulti}
The following statements are equivalent:
\ben
\item\label{2consecmultc1} $G$ does not allow a matrix with the SSP that has two consecutive multiple eigenvalues;
\item\label{2consecmultc2} $G$ does not allow a matrix with the SMP that has two consecutive multiple eigenvalues;
\item\label{2consecmultc3} $G$ does not contain a minor in the family 
\[\mathcal{F}'_2=\{K_3\dcup K_3, K_3\dcup K_{1,3}, K_{1,3}\dcup K_{1,3}, C_4, {\rm Campstool}, H\mbox{-}{\rm tree}\};\]
\item\label{2consecmultc4} $G$ is a disjoint union of $G_1$ and any number of paths, where $G_1$ is either a generalized star or a generalized $3$-sun.
\een
\end{thm}
\begin{proof}
For graphs $G$ in $\mathcal{F}'_2$, there is a matrix $A\in \mptn(G)$ with the SSP and having consecutive multiple eigenvalues (see Table \ref{tab11minminor}).  Then by Theorem \ref{minormon}, every graph that contains a $\mathcal{F}'_2$ minor allows a matrix with the SSP and having consecutive multiple eigenvalues.  Such a matrix also has the SMP.  In other words, \eqref{2consecmultc2} $\Rightarrow$ \eqref{2consecmultc1} $\Rightarrow$ \eqref{2consecmultc3}.

To see that \eqref{2consecmultc3} $\Rightarrow$ \eqref{2consecmultc4}, assume $G$ is a graph that does not contain a $\mathcal{F}'_2$ minor.  Since $G$ does not have any of $K_3\dcup K_3,$ $ K_3\dcup K_{1,3}, $ or $K_{1,3}\dcup K_{1,3}$ as a minor, every component of $G$ but one is a path.  Since paths do not allow multiple eigenvalues and the SSP guarantees eigenvalues from different components are not the same, we may assume $G$ is connected.  Since $G$ does not have any of $C_4,K_3\dcup K_3,$ or $ \text{Campstool}$ as a minor,   $G$ is either a tree or a unicyclic graph with a triangle.

Suppose $G$ is unicyclic with a triangle.  Consider $G$ formed by attaching three trees $T_1,T_2,T_3$ to a triangle on vertices $v_1,v_2,v_3$ respectively.  If $T_i$ is not a path with $v_i$ as an endpoint, then $G$ has  the Campstool as a minor.  Thus $G$ is a generalized $3$-sun, and $G$ does not allow a matrix having consecutive multiple eigenvalues by Lemma \ref{3sunlem}.

Suppose $G$ is a tree.  If $G$ has two vertices of degree at least 3, then $G$ has the $H$-tree as a minor.  So $G$ must be a generalized star, which does not allow a matrix having consecutive multiple eigenvalues by Lemma \ref{collision}.

The fact that \eqref{2consecmultc4} implies  \eqref{2consecmultc2} is obvious.
\end{proof}

\begin{cor}
A graph $G$ allows a matrix with two consecutive multiple eigenvalues and the SSP if and only if $G$ allows a matrix with two consecutive multiple eigenvalues and the SMP if and only if $G$ has a minor in the family $\mathcal{F}'_2=\{K_3\dcup K_3, K_3\dcup K_{1,3}, K_{1,3}\dcup K_{1,3}, C_4, \text{Campstool}, H\text{-tree}\}$.
\end{cor}

If a graph $G$ does not allow a matrix with two consecutive multiple eigenvalues, then $G$ does not allow a matrix with the SSP and having two consecutive multiple eigenvalues.  Thus, we have the following corollary to Theorem \ref{thm:twoconmulti} and Lemma \ref{3sunlem}.

\begin{cor}
A connected graph $G$ does not allow a matrix with two consecutive multiple eigenvalues if and only if $G$ is a generalized star or a generalized $3$-sun.
\end{cor}

\section{Minor monotonicity of the SSP and SMP}\label{sminormon}


It is known that the maximum nullity of a graph is not minor monotone (in fact, not even subgraph monotone \cite[Example 5.1]{BFH}). However, 
under the additional assumption of having the SAP  the maximum nullity is minor monotone; that is, if $G$ is a minor of the graph $H$
and there is a matrix in $\mathcal{S}(G)$ with nullity $k$ and the SAP, then there is a matrix in $\mathcal{S}(H)$ with nullity $k$ 
and the SAP \cite{BFH3}. 

In this section we study  minor monotonicity  with respect to spectra and multiplicity lists 
and the SSP and SMP properties.   Note that the monotonicity of SSP and SMP for $G$ a  subgraph of $H$ was established in \cite{genSAP}.  However,  contraction is more subtle.    When $G$ is obtained from $H$ by a contraction, it may be necessary to restrict the simple eigenvalue added to being  at one of the extreme ends of the spectrum.  For example, $C_4$ is a minor of $C_5$, $C_4$ realizes (2,2) with the SSP, but  $C_5$ cannot realize (2,1,2) by Lemma \ref{firstlast2}, yet $C_5$  can  realize (2,2,1).  
{Much of our discussion will focus on the SSP since the proofs for the SMP are analogous. }

\subsection{Tangent spaces of pertinent manifolds}

For an $n\times n$ symmetric matrix $M$, let $\vect(M)$ be the vector of dimension $\binom{n+1}{2}$ with entries from the upper triangular part of $M$; $\vect(M)$ is indexed by $(i,j)$ in  lexicographic order for $1\leq i\leq j\leq n$.  For a set $E$ of pairs $(i,j)$ with $1\leq i\leq j\leq n$, $\vect_E(M)$ is the subvector of dimension $|E|$ that contains only the entries corresponding to indices in $E$.
In the following, $E_{ij}$ denotes the $n\times n $ matrix with a $1$ in position $(i,j)$ and $0$ elsewhere, and
$K_{ij}$ denotes the $n\times n$ skew-symmetric matrix  $E_{ij}-E_{ji}$.

\begin{defn}\label{TSmatrices}  Let $M$ be an $n\times n$ symmetric matrix.  
\begin{enumerate}
\item The \emph{SSP tangent space matrix} $\TSS(M)$ of $M$ is the $\binom{n+1}{2}\times\binom{n}{2}$ matrix such that the $(k,\ell)$-column is $\vect(MK_{k\ell}+K_{\ell k}M)$ and the columns are indexed by $(k,\ell)$  in  lexicographic order for $1\leq k<\ell\leq n$.  
\item The \emph{SMP tangent space matrix} $\TSM(M)$ of $M$ is the $\binom{n+1}{2}\times\left(\binom{n}{2}+q\right)$ matrix obtained by augmenting $\TSS(M)$ with the $q$ columns $\vect(M^k)$ for $k=0,1,\ldots, q-1$ (where $q$ is the number of distinct eigenvalues of $M$).
\item The \emph{SAP tangent space matrix} $\TSA(M)$ of $M$ is the $\binom{n+1}{2}\times n^2$ matrix such that the $(k,\ell)$-column is $\vect(ME_{k\ell}+E_{\ell k}M)$ and the columns are indexed by $(k,\ell)$ with $1\leq k,\ell\leq n$ in the lexicographic order.
\end{enumerate}
\end{defn}

By \cite[Theorem~27]{genSAP}, the column spaces of the tangent space matrices $\TSS(A),\TSM(A),\TSA(A)$ are the corresponding tangent spaces of the manifolds $\mei_A,\mmu_A,\mrk_A$, respectively; we have followed the notation in \cite{genSAP}.

\begin{rem}
\label{spacetvl}
Let $A$ be a symmetric matrix and $\mei_A$ its isospectral manifold.  Let $V$ be a subspace of symmetric matrices.  Then $\mei_A$ and $V$  intersect transversally at $A$ if and only if the zero matrix is the only matrix $X$ such that $X\in V^{\perp}$ and $ \vect(X)\trans \TS(A)=\bzero\trans$.  In particular, if $G$ is the graph of $A$ and $V$ is the subspace of symmetric matrices whose nonzero entries correspond to only edges or diagonal entries, then we can see that $A$ has the SSP if and only if the rows of $\TSS(A)$ corresponding to nonedges are linearly independent.
\end{rem}






We now focus on the  tangent space to the isospectral manifold. 
We first compute  the tangent space matrix for a matrix of the form $A \oplus [\lambda]$, which is denoted by $A_{\lambda}$.
Given a matrix $M$, $\max(M)$ denotes the maximum absolute value of an entry of $M$, and 
 $\bm_j$ denotes the $j$th column of $M$.

\begin{lem}
\label{TS-sum}
Let $A$ be an $n\times n$ matrix and $\lambda$ be a real number.
Then,  after the columns indexed by $(i,n+1)$ have been permuted to the right and the rows indexed by $(i,n+1)$ have been permuted to the bottom, $\TS(A_{\lambda})$
has the form
 \[
\left( \begin{array}{c|c}
\TS(A) & O \\ \hline
O & A   -\lambda I_n 
\\ \hline
\begin{array}{ccc} 0 & \cdots & 0 \end{array} & \begin{array}{ccc} 0 & \cdots & 0 \end{array} 
 \end{array} 
\right).
\]
\end{lem}

\begin{proof}
The result follows from the following computations.
For $i,j \in \{1,2, \ldots, n\}$,   
 \[
A_{\lambda}K_{ij} - K_{ij}A_{\lambda}=\left( \begin{array}{c|c} 
AK_{ij} - K_{ij}A& \begin{array}{c} 0 \\ \vdots \\ 0 \end{array} \\  \hline
\begin{array}{ccc} 0 & \cdots & 0 \end{array}  &  0 \end{array} \right).
\]
For $i\in \{1,2,\ldots, n\}$,

 \[
A_{\lambda}K_{i,n+1} - K_{i,n+1}A_{\lambda}= \left(\renewcommand{\arraystretch}{1.3} \begin{array}{c|c}
O & \ba_i-\lambda \be_i  \\ \hline 
 \ba_i\trans-\lambda \be_i\trans   & 0 
\end{array} \right),
\]
 where $\ba_i$ and $\be_i$ denote  the $i$th column of $A$ and the $i$th standard basis vector, respectively.  \end{proof}
 
We now use Lemma \ref{TS-sum} to give a perturbation result for $A_{\lambda}$. Specifically, if $\lambda$ is large enough, you can perturb an SSP matrix $A_\lambda$  by any sufficiently small $E$ that is combinatorially orthogonal to $A_\lambda$, and find a correction matrix $F$ whose size is controlled such that $\G(F)$ is a subgraph of $\G(A_\lambda)$ and $\spec(A_\lambda+E+F)=\spec(A_\lambda)$.  This allows us to add edges between the isolated vertex represented by $\lambda$ and $\G(A)$ while preserving the spectrum and controlling the size of the modification.  As usual, $\sigma_k(M)$ denotes the smallest singular value of a $k\x k$  matrix $M$.

\begin{lem}
\label{epsilon-bound}
Let $A$ be an $n\times n$ symmetric matrix with the SSP and graph $G$. 
There exists a $\Delta>0$ such that for all  sufficiently small  $\epsilon >0$ and all  sufficiently large $\lambda$, 
the following holds:
\begin{itemize}
\item[] For each symmetric matrix $E$ of order $n+1$  such that $E \circ I=O$ and $E\circ A_{\lambda}=O$ 
and $\max(E)\leq \epsilon$, there there exists a symmetric $F=[f_{ij}]$  of order $n+1$
such that $f_{ij} \neq 0$ only if $i=j$ or $ij$ is an edge of $G$,  $\max (F) \leq \Delta \epsilon$, and $A_{\lambda}+E+F$ is cospectral with $A_{\lambda}$.
\end{itemize}
\end{lem}

\begin{proof}
For $\lambda $ sufficiently large,  $\lambda$ is not an eigenvalue of $A$, so $A_{\lambda}$ has the SSP, and  $\G(A_{\lambda})=G\dcup K_1$. 
Let $\tau$ be the indices of the rows of $\TS(A_{\lambda})$ corresponding to 
$(i,j)$ where $i\neq j$ and $ij$ is not an edge of $\G(A_{\lambda})$.
Since $A_{\lambda}$ has the SSP, the rows of $\TS(A_{\lambda})$ corresponding to $\tau$ are linearly independent
and  there exists  an  invertible $|\tau|\x |\tau|$ submatrix $\TS(A_{\lambda})[\tau, \mu]$ of  $\TS(A_{\lambda})$.  

We are building toward a proof that the SSP is preserved for decontractions.  The proof that the SSP is preserved for supergraphs \cite[Theorem~10]{genSAP} makes use of the Implicit Function Theorem.  For this result, we need a different form of the Implicit Function Theorem, which uses the invertibility of $\TS(A_{\lambda})[\tau, \mu]$ to guarantee uniform continuity within some neighborhood.  This yields, for $\epsilon>0$ sufficiently small and a given $E$ with $\max(E)\leq\epsilon$, a positive $\Delta$ (independent of $\epsilon$) and $F$ satisfying the conditions stated in the lemma, for which $A_\lambda+E+F$ cospectral with $A_\lambda$.
\end{proof}


Next we compute the tangent space of a special type of perturbation of $A_{\lambda}$. 

\begin{lem}
\label{TSper}
Given $A_\lambda$ and a matrix $L$ with $\max(L)\leq 1$, suppose 
\[ C=A_{\lambda} + \left(\renewcommand{\arraystretch}{1.3} \begin{array}{c|c} O & \bx \\ \hline \bx\trans &  0\end{array} \right) + O(\epsilon)L. \]
Then, after permuting as in Lemma \ref{TS-sum}, $\TS(C)$ 
has the form
\[
\left(\renewcommand{\arraystretch}{1.3}
\begin{array}{c|c}
\TS({A}) & V_1\\ \hline
V_2 &  A- \lambda I + V_3\\ \hline
\bzero\trans & 2\bx\trans 
\end{array} 
\right) +O(\epsilon)\TS(L)=\left(\renewcommand{\arraystretch}{1.3}\begin{array}{c|c}
\TS({A}) & V_1\\ \hline
V_2 &  A- \lambda I + V_3\\ \hline
\bzero\trans & 2\bx\trans 
\end{array} 
\right) +O(\epsilon) M
\]
where
$\max (V_1), \max (V_2), \max (V_3)$ are all at most $ 2 \max (\bx) $, and $M$ is a fixed matrix.
\end{lem}
 
\begin{proof}
The form of $C$ implies 
\[ \TS(C) = \TS(A_{\lambda}) + \TS \left( \renewcommand{\arraystretch}{1.3}
\left( \begin{array}{c|c} O & {\bf x} \\ \hline \mathbf{x}\trans &  0\end{array} \right) \right) 
+ O(\epsilon)\TS(L).
\]
The result then follows from the following computations.
For $i,j \in \{1,2, \ldots, n\}$, 
\[  \left( \renewcommand{\arraystretch}{1.3}\begin{array}{c|c} O & \bx \\ \hline \bx\trans &  0\end{array} \right)K_{ij}
+ K_{ji}\left(\renewcommand{\arraystretch}{1.3} \begin{array}{c|c} O & \bx \\ \hline \bx\trans &  0\end{array} \right)
= \left(\renewcommand{\arraystretch}{1.3} \begin{array}{c|c}
O & x_i{\be_j}-x_j {\be_i}\\ \hline 
 x_i{\be_j}\trans-x_j {\be_i}\trans & 0 \end{array} \right)
\!.\]
For $i\in \{1,2,\ldots, n\}$, 
\[ \left(\renewcommand{\arraystretch}{1.3} \begin{array}{c|c} O & \mathbf{x} \\ \hline \mathbf{x}\trans &  0\end{array} \right)K_{i,n+1} + 
K_{n+1,i}\left(\renewcommand{\arraystretch}{1.3} \begin{array}{c|c} O & \mathbf{ x} \\ \hline \mathbf{x}\trans &  0\end{array} \right)= {\left(\renewcommand{\arraystretch}{1.3} \begin{array}{c|c} -{\be_i}\bx\trans -\bx{\be_i}\trans & O \\ \hline 
O & 2x_i \end{array} \right)\!.} \qedhere
\]
\end{proof}

Although it is not part of the statement of the next result, we are thinking of a graph $G$ of order $n$  being used to create a graph $H$ of order $n+1$ with the partition of   the neighborhood of $n$ in $G$, $N_G(n)$, into $\alpha$ and $\beta$ representing sorting the neighbors of $n$ in $G$ into those incident with {$n$ in $\alpha$ versus those incident with $n+1$ in $\beta$}. The existence of  a matrix of  the special form described in Theorem~\ref{cor:form} allows this operation while preserving the spectrum.   

\begin{thm}
\label{cor:form}
Let $A$ be an $n \times n$ SSP matrix with graph $G$, $\lam$ {sufficiently large}, 
and $\alpha\dcup \beta$ be a partition of  $N_G(n)$.
Suppose that for a matrix $L$ with $\max(L)\leq 1$, there is a symmetric $(n+1)\x (n+1)$ matrix 
\[ C=A_{\lambda} + \left(\renewcommand{\arraystretch}{1.3} \begin{array}{c|c} O & \mathbf{x} \\ \hline \mathbf{x}\trans &  0\end{array} \right)+O(\epsilon)L,\]
such that {$\spec(C)=\spec(A_{\lam})$,} 
\[ {   \begin{array}{ll} c_{n,j}=c_{n+1,j}= 0 & \mbox{for $j \not\in N_G[n] \cup \{ n+1\}$}, \\
  c_{n,j}=c_{n+1,j}  & \mbox{for  $j\in \alpha\cup\{n\}$}, \\
   c_{n,j}=-c_{n+1,j} & \mbox{for $j\in \beta$,}\\
    c_{n,n}\ne c_{n+1,n+1}, & 
 
    \end{array} }  \]
 and $C(n+1)$ has graph $G$.  Let  $V$ denote the span of 
the matrices  
\[ 
\begin{array}{cl}
E_{ii}& \mbox{ for }i=1,2,\ldots, n+1,\\ 
 E_{ij}+ E_{ji} &  \mbox{ for $ij$ an edge of $G$ not incident to $n$},\\
  E_{nj}+ E_{jn} - E_{n+1,j}-E_{j,n+1} & \mbox{ for }j\in \alpha\cup \{n\},\mbox{ and}\\
  E_{nj}+E_{jn}+ E_{n+1,j}+ E_{j,n+1}  &\mbox{ for }j\in \beta.
\end{array} 
\]
Then  $\mei_C$ intersects $V$ transversely.
 Moreover, the matrix $RCR\trans$, where 
 \[
 R=I_{n-1} \oplus \left( \begin{array}{rr} \sqrt{2}/2 & \sqrt{2}/{2} \\
 -\sqrt{2}/{2} & \sqrt{2}/{2} \end{array} \right), 
 \]
 is cospectral with $A_{\lambda}$, has the SSP, and its graph  $H$ is the graph obtained from $G$ by inserting a new vertex $n+1$, and edge between $n$ and $n+1$, and 
for  $j\in {\beta}$ replacing the edge $nj$ in $G$ with   the edge joining $j$ and $n+1$. \end{thm}

\begin{proof}  We consider a vector $\bw$ in the left-nullspace of $\TS(C)$, and for convenience 
index its elements by $(i,j)$ in lexicographic order with $1\leq i \leq j \leq n+1$. 
 The claimed transversality follows by showing that  
 if $\bw_{(i,i)}=0$  for all $i$,  $\bw_{(i,j)}=0$ when $ij$ is an edge of $G$ not incident to $n$, 
 $\bw_{(j,n)} =\bw_{(j,n+1)}$ for $j\in \alpha\,{\cup \{n\}}$, and $\bw_{(j,n)} =-\bw_{(j,n+1)}$ for  $j\in \beta$, then $\bw$ is the zero vector.
 
 Let $\widehat{C}$ be the matrix obtained from $\TSS(C)$ by replacing the row indexed by 
 $(j,n+1)$ by the sum of the rows indexed by $(j,n+1)$ and $(j,n)$ for $j\in \alpha\,{\cup \{n\}}$,
 and the difference of the rows for $j\in \beta$.  
 Note that the desired  transversality is implied by the linearly independence of rows of $\widehat{C}$ 
 indexed by $\gamma=\{ (i,j) : 1\leq i< j\leq n+1\mbox{  and } ij \mbox{ not an edge of }G\}$ and the rows indexed by {$(j,n+1)$ for $j=1, \ldots , n$}.
By Lemma~\ref{TSper}, after the usual permutation, the matrix formed by these rows has the form 
\beq \left(\renewcommand{\arraystretch}{1.3} \begin{array}{c|c}
\TS(A)[\gamma,:] & V_1[\gamma,:]  \\ \hline
{V_2'}& {A-\lambda I_n+ V_3'} \end{array} 
\right) + O(\epsilon)M
\label{TSeq}
\eeq
for  a fixed matrix $M$. 
Since $\lambda$ is significantly larger than any entry of $V_1$, $V_2'$, and $V_3'$, and 
the corresponding rows of $\TS(A)$ are linearly independent by the fact that $A$ has the SSP, the matrix in \eqref{TSeq} has  linearly independent rows.

Note that $RCR\trans$ is the matrix obtained from $C$ by performing a Givens rotation of $\pi/4$ 
on the last two rows and columns of $C$.  Hence the submatrices of $RCR\trans$ and $C$ lying in the first $n-1$ rows and columns agree.
The hypothesis on the last two rows of $C$ imply that $RCR\trans$ has the given graph.

Now suppose that $X$ is an $(n+1) \times (n+1)$ matrix 
with $X\circ I=O$,  $X\circ RCR\trans=O$, and $[X,RCR\trans]=O$. Note in particular, $X[ \{n,n+1\}, \{n, n+1\}]=O$ since $(RCR\trans)_{n,n+1}\ne 0$.
It follows that $RXR\trans $ is in the  orthogonal complement of ${V}$ in the space of symmetric matrices.  As ${V}$ intersects $\mei_C$ transversely, we conclude that $RXR\trans=O$.  Thus, $X=O$, and we conclude  that $RCR\trans$ has the SSP.
\end{proof}

\subsection{Proof of minor monotonicity for SSP}

In this section we prove that if $A$ has the SSP and $\lambda$ is sufficiently large, 
then there exists a matrix $C$ satisfying the hypothesis of Theorem~\ref{cor:form}.  We begin with a needed perturbation result.
\begin{lem}
\label{bound}
Let $M$ be a positive definite $n\times n$ matrix with smallest eigenvalue  $\lambda_{\min}(M)$ and $\mathbf{b}$ be an $n$-vector.
If $\mathbf{x}$ is the solution to $M\mathbf{x}=\mathbf{b}$, then 
\[\max(\mathbf{x}) \leq \frac{\sqrt{n}}{\lambda_{\min}(M) }\cdot \max (\mathbf{b}). \]
\end{lem}

\begin{proof}
Let $\bm_j$ denote the $j$th column of $M^{-1}$, and observe that $\|\bm_j\| \le \lam_{\max}(M^{-1})=\frac{1}{\lambda_{\min}(M)}$.  Since $\bx=M^{-1}\bb$,   the Cauchy-Schwarz inequality yields
\begin{eqnarray*}
|x_j|& \le & 
\| \bm_j \| \cdot \|   \bb \|\\
& \leq  &\frac{1}{\lambda_{\min}(M)} \cdot \| \mathbf{b}\| \\ 
& \leq  &\frac{1}{ \lambda_{\min}(M)} \cdot   \sqrt{n} \cdot \max(\mathbf{b}). \qquad\qquad\qedhere
\end{eqnarray*}
\end{proof}


\begin{prop}
\label{tech}
Let $A$ be a symmetric $n\times n$ matrix with graph $G$.   Let $\epsilon>0$ with $\epsilon\le\max(A)$, $\Delta>0$, and $\lambda \geq
  \max \left\{ 2\rho(A), \frac{4n (\max (A)(1+ \Delta))^2}{\epsilon^2} \right\}$.
Let $B= A_\lambda+ E + F$ be an $(n+1)\times (n+1)$ matrix  such that 
\begin{itemize} 
\item[(a)] $E=[e_{ij}]$ is a symmetric matrix with $\max (E) \leq \epsilon$, $E\circ A_\lambda=O$, and $E\circ I=O$; and 
\item[(b)] $F=[f_{ij}]$ is a symmetric matrix with $\max(F)\leq \Delta \max(E)$ and  $f_{ij} \neq 0$ only if $ij$ is an edge of $G$ or $i=j$. 
\end{itemize}
Let $\alpha \dcup \beta$ be a partition of $N_G(n)$. 
Then there is an orthogonal matrix $Q$ 
  such that 
\begin{itemize}
\item[(c)]
for all $i$ and $j$ except $\{i,j\} \in \{ \{n+1,k\}: k \in N_G[n] \}$, the $ij$-entry of $Q\trans BQ-B$ has absolute value at most $O( \epsilon^2)$.
\item[(d)] For each $j$ in $\alpha\cup\{n\}$,  $(Q\trans BQ)_{n+1,j}-(Q\trans BQ)_{nj}- e_{n+1,j} $ has absolute value at most $ O( \epsilon^2)$, and 
\item[(e)]
for each $j$ in $\beta$, $(Q\trans BQ)_{n+1,j}+(Q\trans BQ)_{nj} - e_{n+1,j} $ has absolute value at most $O( \epsilon^2)$.
\end{itemize}
\end{prop}

\begin{proof}  
Since  $\lambda \ge 2  \rho(A)$,   $\lambda I-A$ is positive definite.
Define $\mathbf{k}$ to be the $n$-vector given by $(\lambda I-A)\mathbf{k} = D(\ba_{n}+\bff_{n})$ where $\ba_n$ and $\bff_n$ denote  the $n$th columns of $A$ and $F(\{n+1\})$, and $D=\mbox{diag}(d_1, \ldots, d_n)$ is the diagonal 
matrix with  $d_i$ equal to  $1$  if $i \in \alpha\cup\{n\}$, $-1$ if $i\in \beta$ and $0$ otherwise. 

Note  that the  choice of $\lambda$ and Lemma \ref{bound} imply that 
\begin{eqnarray*}
\lambda \max(\mathbf{k})^2  &\leq & \lambda \left( \frac{\sqrt{n} \max (A)(1 + \Delta)}{\lambda_{\min}(\lambda I-A)} \right)^2  
\\
& \leq &\frac{\lambda n (\max(A)(1+ \Delta))^2}{(\lambda-\rho(A))^2} 
\\ 
& \leq & \frac{4\lambda n (\max(A)(1+ \Delta))^2}{\lambda^2}\\
& \leq &  \epsilon^2 .
\end{eqnarray*}
Let  $K$ be the skew-symmetric matrix
\[\renewcommand{\arraystretch}{1.3} 
\left( 
\begin{array}{c|c} O & -\mathbf{k} 
\\  \hline
\mathbf{k}\trans & 0 \end{array} \right)
\]
and let $Q$ be the matrix exponential $e^{K}$, which is an orthogonal matrix.
For $\epsilon$ sufficiently small, 
 $Q\trans BQ$ has the form 
\[
B+  A_{\lambda} K - KA_{\lambda} 
+ O(\epsilon^2)L 
\]
with $\max(L)\leq 1$, because every entry of $K$ is $O(\epsilon)$, as is every entry of $B-A_\lam$. 
Since \begin{eqnarray*}
A_{\lambda} K - KA_{\lambda} &=&
 \left(\renewcommand{\arraystretch}{1.3}  \begin{array}{c|c} O & D(\mathbf{a}_n+\mathbf{f}_n) \\ \hline
(\mathbf{a}_n+\mathbf{f}_n)\trans D  &O  \end{array} \right)\!,
\end{eqnarray*}
 $Q\trans BQ$ has the desired form. 
\end{proof}

\begin{thm}
\label{thm:fixpoint}
Let $A$ be a symmetric $n\times n$ matrix with graph $G$ and the SSP,  and let $\alpha\dcup \beta$ be a partition 
of $N_G(n)$. 
Let $H$ be the graph obtained from $G\dcup K_1$ by joining
$n+1$ to each vertex in $N_G[n]$.  
Then 
for $\epsilon>0$ sufficiently small and $\lambda$ sufficiently 
large there is a matrix $C$ such that: 
\begin{itemize}
\item the spectrum of $C$ is that of $A$ along with $\lambda$; 
\item $|c_{ij}-a_{ij}|\leq O(\epsilon)$  for all $i$ and $j$
with $\{i,j\} \notin \{\{n,j\}: j \in {N_G[n]}\}$;
\item $c_{ni}=c_{n+1,i}$ for $i\in \alpha\cup\{n\}$;
\item $c_{ni}=-c_{n+1,i}$ for $i \in \beta$; 
\item $c_{n,n}\ne c_{n+1,n+1}$; and 
\item the graph of $C$ is $H$.
\end{itemize}
\end{thm}

\begin{proof}
  Let $\Omega$ be the set of $(n+1)\times (n+1)$ symmetric matrices defined as 
\[\Omega= \{E: \max(E)\leq \epsilon, E \circ A_{\lambda}=O,\text{ and }E \circ I=O\}. \]
 Since $A$ has the SSP, by Lemma~\ref{epsilon-bound}  there is  
 a $\Delta>0$ such that for $\epsilon>0$ sufficiently small (we also require $\epsilon\leq \max(A)$) and $\lam$ sufficiently large,  for all $E\in \Omega$   there is a symmetric matrix $F$ such that the graph of $F$ is a subgraph of $G\dcup \{n+1\})$,  $\max(F) \leq \Delta \max(E)$, and $A_{\lambda} +E +F$  is cospectral to $A_{\lambda}$.  As in Lemma~\ref{epsilon-bound}, $F$ can be chosen to be a uniformly continuous function of the entries of $E$.  Denote such $A_{\lambda} +E+F$ by $B_{E}$.
 
With $\lam$ sufficiently large,  $(B_E)_{nn}\ne (B_E)_{n+1,n+1}$.  By Proposition \ref{tech}, there exists a $Q$ such that $Q\trans B_{E} Q$ satisfies (c), (d), and (e).
 Let $\phi(E)$ be the $(n+1) \times (n+1)$ symmetric matrix with its lower triangular part defined as 
 \[
 \phi(E)_{ij} {=} \left\{ \begin{array}{cl} (Q\trans B_{E}Q)_{ij} & \mbox{ if $i\neq j$, $ij$ is not an edge of $H$}\\
 (Q\trans B_{E}Q)_{n+1,j}-(Q\trans B_EQ)_{nj}& \mbox{ if $i=n+1$ and $j \in \alpha\cup\{n\}$} \\
 (Q\trans B_{E}Q)_{n+1,j}+(Q\trans B_EQ)_{nj}& \mbox{ if $i=n+1$ and $j \in \beta$} \\
 0 & \mbox{otherwise.} 
\end{array} \right.
\] 
 Then (c), (d), and (e) of Proposition \ref{tech}, imply that 
 \[ 
 |\phi(E)_{ij} -e_{ij}| \leq O( \epsilon^2) \mbox{ for all $ij$ with $i\neq j$ and $ij$ not an edge of $G \dcup \{ n+1 \}$},
 \]
 where $E=[e_{ij}]$.
We claim that there exists an $E$ such that $\phi (E) = O$.
Suppose to the contrary that $\phi(E)\neq O$ for all $E$.
Then $f: \Omega \rightarrow \Omega$ by 
$f(E)= -\epsilon \frac{\phi(E)}{\max \phi (E)}$
is a well-defined, continuous map. Let $ij$ be an index with $|\phi(E)_{ij}|$ largest.
Note that  $e_{ij}$ and $f(E)_{ij}$ have opposite signs unless $|e_{ij} | \leq O(\epsilon^2)$, and 
in the latter case $| f(E)_{ij}|= \epsilon > |e_{ij}|$.  Thus $f$ has no fixed point.
However, $\Omega$ is homeomorphic to a closed ball in $\mathbb{R}^d$, where $d$ is the number of edges 
not in $G\dcup \{ n+1\}$.  So the nonexistence of a fixed point would contradict the Brouwer Fixed-Point Theorem.

Thus there exists $E \in \Omega$ such that $\phi(E)=O$.  For such $E$, 
$B_E$, and hence $Q\trans B_{E}Q$, is cospectral with $A_{\lambda}$.  Let $C=Q\trans B_{E}Q$.  Then $C$ has the desired properties.

\end{proof}

Applying Theorem~\ref{cor:form} to the matrix $C$ found in Theorem~\ref{thm:fixpoint}, we have the following result. 

\begin{lem} [Decontraction Lemma for SSP]  \label{decontractSSP}
Suppose  $G$ is obtained from $H$ by contraction of a single edge whose endpoints have disjoint neighborhoods, and $A\in\mptn(G)$ with the SSP.  Then for $\lambda$ sufficiently large,  there is an SSP matrix $A'\in \mptn(H)$ with the same eigenvalues as $A$ and the additional eigenvalue  $\lambda$.  \end{lem}

Similar arguments can be employed to establish the following analogous result for SMP.

 \begin{lem} [Decontraction Lemma for SMP] \label{decontractSMP}
Suppose  $G$ is obtained from $H$ by contraction of a single edge whose endpoints have disjoint neighborhoods, and $A\in\mptn(G)$ with SMP.  Then there is an SMP matrix $A'\in \mptn(H)$ with the  ordered multiplicity list obtained  by adding a 1 to the end  of  $\oml(A)$ .\end{lem}

Combining these results, and  subgraph results \cite[Theorem~36]{genSAP}, we obtain the following general result 
regarding graph minors.

\begin{thm}[Minor Monotonicity Theorem] \label{minormon}
Suppose  $G$ is a minor of  $H$ obtained by contraction of $r$ edges,  deletion of $s$ vertices, and deletion of any number of edges, and $A\in\mptn(G)$.  

If $A$ has SMP and   $\oml(A)=(m_1,\dots,m_t)$, then  there is a matrix $A'\in \mptn(H)$ having SMP with $\oml(A')$ obtained  from $\oml(A)$ by adding $r+s$ ones,  with at most $s$ of these between $m_1$ and $m_t$. 

If in addition $A$ has the SSP, then $A'$ can be chosen to have the SSP, $\spec(A)\subseteq\spec(A')$, all eigenvalues not in $\spec(A)$ are simple and distinct,  at most $s$ of these additional eigenvalues are between 
$\lambda_{\min}(A)$ and $\lambda_{\max}(A)$,  and $s$  simple eigenvalues (including all of those between 
$\lambda_{\min}(A)$ and $\lambda_{\max}(A)$)  can be chosen arbitrarily. 
\end{thm}


\section{The Matrix Liberation Lemma and other technical tools}\label{sMLL}
In this section we prove the Matrix Liberation Lemma and some consequences, which were used to establish several previous results.  In each case there is a rectangular matrix that characterizes the extent to which   the {zero} entries of the matrix can be perturbed  (with sufficiently small changes) 
while preserving the exact spectrum (SSP),  the ordered multiplicity list  (SMP), or  the rank (SAP).  These matrices, which are necessary in order to state the Matrix Liberation Lemma, are defined Definition \ref{vermtxdef}.  

\begin{defn}\label{vermtxdef}  Let $G$ be a graph and let $\overline{E}$ be the set of pairs $\{(i,j):i<j,\{i,j\}\in E(\overline{G})\}$.  Let $A\in\mptn(G)$ and $p=|\overline{E}|$.  
\begin{enumerate}
\item The {\em SSP verification matrix} $\verS(A)$ of $A$ is the $p\times\binom{n}{2}$ matrix $\TSS(A)[\Ec,:]$.
\item The {\em SMP verification matrix} $\verM(A)$ of $A$ is the $p\times\left(\binom{n}{2}+q\right)$ matrix $\TSM(A)[\Ec,:]$.
\item The {\em SAP verification matrix} $\verA(A)$ of $A$ is the $p\times n^2$ matrix $\TSA(A)[\Ec,:]$.
\end{enumerate}
\end{defn}

It was shown in \cite{genSAP} that $A$ satisfies the SSP, SMP, or SAP, respectively,
if and only if the matrix $\verS(A)$, $\verM(A)$, or $\verA(A)$
has full rank $p$.  Notice that in $\verM(A)$ the columns  $\vect_{\Ec}(A^0)$ and $\vect_{\Ec}(A^1)$ are always zero, so we may omit them for verifying if $A$ has the SMP or not.

The {\em support} of a vector $\bx$ is the set of indices of nonzero coordinates of $\bx$, and is denoted by $\supp(\bx)$.  The following result from \cite{HLS} will be used.

\begin{lem}\label{HLStangent}{\rm \cite[Corollary 2.2]{HLS}}  
Assume that $M_1,\ldots, M_k$ intersect transversally at $\by$, and let $\bv$ be a common
tangent to each of $M_1,\ldots,M_k$ with $||\bv||=1$. Then for every $\epsilon>0$ there exists a point $\by'\neq \by$ such that $M_1,\ldots,M_k$ intersect transversally at $\by'$, and
\[\left\|\frac{1}{\|\by-\by'\|}(\by-\by')-\bv\right\|<\epsilon.\]
\end{lem}

\begin{lem} [Matrix Liberation Lemma] \label{mtxliblem}
Let $G$ be a graph and $A\in\mptn(G)$.
Let $\Psi$ be one of the following:
\begin{itemize}
\item Case 1. $\Psi = \verS(A)$.
\item Case 2. $\Psi = \verM(A)$.
\item Case 3. $\Psi = \verA(A)$.
\end{itemize}
Suppose $\bx$ is a vector in the column space of $\Psi$ such
that the complement of $\supp(\bx)$ corresponds to a linearly independent set of rows in $\Psi$. Let
$H$ be a spanning subgraph of $\Gc$ whose edges correspond to $\supp(\bx)$.  Then $A$ can be perturbed to
$A'\in\mptn(G \cup H)$ such that:
\begin{itemize}
\item Case 1. $A'$ satisfies the SSP with the same spectrum as $A$.
\item Case 2. $A'$ satisfies the SMP with the same ordered multiplicity list as $A$.
\item Case 3. $A'$ satisfies the SAP with the same rank as $A$.
\end{itemize}
\end{lem}

\begin{proof}
We prove the result in Case 1, as the other cases follow by similar arguments.
Assume that $\Psi = \verS(A)$.  Let $\Ec=\{(i,j):i<j,\{i,j\}\in E(\Gc)\}$.  Since $\bx$ is in the column space of $\Psi$ and the column space of $\TSS(A)$ is the tangent space of the isospectral manifold $\mei_A$ at $A$, there is a matrix $B$ in the tangent space such that $\vect_{\Ec}(B)=\bx$.  We may scale $B$ such that $\|B\|=1$.

Let $V$ be the subspace (which is a  manifold) of $n\times n$ symmetric matrices whose $i,j$-entry is zero if $\{i,j\}\in E(\overline{G\cup H})$.  Then $V^\perp$ contains matrices whose nonzero entries appear only  at those pairs that correspond to $E(\overline{G\cup H})$, which is also the complement of $\supp(\bx)$.  By our assumption, the set of rows in $\TSS(A)$ corresponding to the complement of $\supp(\bx)$ is linearly independent.  By Remark \ref{spacetvl}, $\mei_A$ and $V$ intersect transversally at $A$.  Also, $B$ is a common tangent to $V$ and $\mei_A$.

Applying Lemma \ref{HLStangent} with two manifolds $V$ and $\mei_A$, $\by=A$, and $\bv=B$, for every $\epsilon>0$ there is a matrix $A'$ with 
\[\left\|\frac{1}{\|A-A'\|}(A-A')-B\right\|<\epsilon.\]
such that $\mei_A$ and $V$ intersect transversally at $A'$.

We may pick $\epsilon$ small enough such that the nonzero entries of $B$ do not vanish in $\frac{1}{\|A-A'\|}(A-A')$, so the entries of $A'$ corresponding to $E(H)$ are nonzero.  Since $A'$ is close to $A$, the nonzero entries of $A$ do not vanish in $A'$, so entries of $A'$ corresponding to $E(G)$ are nonzero.  All these facts and  $A'\in V$ imply $A'\in\mptn(G\cup H)$.  This means $\mei_A$ and $\mptn(G\cup H)$ intersect transversally at $A'$, so $A'$ has the SSP and $\spec(A')=\spec(A)$.
\end{proof}

\begin{rem}  A vector $\bx$ as in the Matrix Liberation Lemma exists if and only if no row of $\Psi$ is zero. In this case,  a vector $\bx$ with every entry nonzero exists in the column span of $\Psi$, the graph $H=\Gc$, and Lemma \ref{mtxliblem} guarantees a matrix with the same spectrum as $A$ such that all off-diagonal entries are nonzero. For the SSP, the condition that   no row of $ \verS(A)$ is zero  is equivalent to $B$ not having a pair of rows that are zero  except for the same diagonal entry.
\end{rem}

We apply the Matrix Liberation Lemma to prove additional results.

\begin{lem}[Augmentation Lemma]\label{shaun}  Let $G$ be a graph on vertices $\{1,\ldots,n\}$ and $A\in\mptn(G)$.  Suppose $A$ has the SSP and $\lam$ is an eigenvalue of $A$ with multiplicity {$k\ge 1$}. Suppose that $\alpha$ is a subset of $\{1,\ldots, n\}$ of cardinality $k+1$ with the property that for every eigenvector $\bx$ of $A$ corresponding to $\lam$, $|\supp(\bx)\cap \alpha |\ge 2$.  
Construct $H$ from $G$ by appending vertex $n+1$ adjacent exactly to the vertices in $\alpha$. Then there exists a matrix $A'\in\mptn( H)$ such that $A'$ has the SSP,  the multiplicity of $\lam$ has increased from $k$ to $k+1$, and other eigenvalues and their  multiplicities are unchanged from those of $A$. \end{lem}

\bpf
Consider $A_{\lambda}= A\oplus [\lambda]$.  By Lemma \ref{TS-sum}, 
$\verS(A_\lambda) $ has the form 
$\verS(A) \oplus (A-\lambda I )$ after suitable permutation of rows and columns.
 The rows of $A-\lam I$ with index not in $\alpha$ (equivalently, the rows of $\verS(A_\lambda) $ indexed by $(j,n+1)$ with $j\not\in\alpha$)  are linearly independent; otherwise there is an eigenvector of $A$ corresponding to $\lambda$ whose support is disjoint from $\alpha$.  From this, the block diagonal structure of $\verS(A_\lambda)$, and the fact that $A$ has the SSP, all the rows of $\verS(A_\lambda) $ associated with non-edges of $H$ are linearly independent.

Let $N$ be a $n \times k$ matrix whose columns are a basis for the null space of $A-\lam I$. 
We show that every order $k$ submatrix of $N[\alpha,:]$ is  invertible: Suppose some  $k\x k$ submatrix $M$ of $N[\alpha,:]$ is  not invertible.  Then there is a vector $\bz$ such that $M\bz=\bzero$.  But then $N\bz$ is an eigenvector of 
$A-\lam I$ whose support intersects $\alpha$ in at most one element, which contradicts
the assumptions. 

Since the rows of the $(k+1)\times k$ matrix $N[\alpha,:]$ are linearly dependent, and each of the $k\times k$ submatrices of $N[\alpha,:]$ is invertible, each coefficient in a dependence relation is nonzero.  This gives a vector $\by$ with $\by\trans N=\bzero$ and $\supp(\by)=\alpha$.
Since $\by$ is orthogonal to each {column} of $N$, $\by$ is in the
row space (which equals the column space) of $A-\lam I$.  Since $\verS(A_\lambda) $ has the form $\verS(A) \oplus (A-\lambda I )$, there is a vector $\hat\by$ in the column space of $\verS(A)$ such that $\supp(\hat\by)$ is those edges $\{j,n+1\}$ with $j\in\alpha$.  The result now follows from Lemma \ref{mtxliblem}.
\epf

\begin{cor}\label{Cn211} For any list of distinct real numbers $\lam_1<\dots<\lam_{n-1}$ and integer $k$ with $1\le k\le n-1$,  there exists a matrix $A\in\mptn(C_n)$ with the SSP such that $A$ has eigenvalues $\lam_k$ with $\mult_A(\lam_k)=2$ and $\mult_A(\lam_i)=1$ for $i\ne k$. \end{cor}
\bpf By \cite[Remark~15]{genSAP}, there is a matrix $A\in\mptn(P_{n-1})$ with the SSP and $\spec(A)=\{\lam_1,\dots,\lam_{n-1}\}$; let $\bv=[v_i]$ be an eigenvector of $A$ with respect to $\lam_k$.    By the structure of $A$ and $(A-\lambda I)\bv=\bzero$, if $v_1=0$ then we can see inductively that $v_2,\ldots ,v_n$ are all zero, so $v_1\ne 0$;  similarly $v_{n-1}\ne0$.  Then by applying Lemma \ref{shaun} with $\alpha=\{1,n-1\}$, there is a matrix $A'\in\mptn(C_n)$ with the SSP and $\spec(A')=\{\lam_1,\dots,\lam_{k-1},\lam_k,\lam_k,\lam_{k+1},\dots,\lam_{n-1}\}$. \epf

The next result  is not required for the rest of the paper, but it gives a different way to compute the verification matrices.  This result displays  the verification matrices as the coefficient matrices of systems of homogeneous equations with the variables on the left (for the traditional view transpose the verification matrices). 
Let $X_{ij}=E_{ij}+E_{ji}$ and $E_o=\{(i,j):1\leq i<j\leq n\}$ the set of off-diagonal pairs.

\begin{prop}\label{normalspace}
Let $G$ be a graph and $A\in\mptn(G)$.  Then 
\begin{enumerate}
\item The $(i,j)$-row of $\verS(A)$ is $\vect_{E_o}(AX_{ij}-X_{ij}A)$.
\item The SMP verification matrix $\verM(A)$ is by augmenting $\verS(A)$ with $q$ columns $\vect_{\Ec}(A^k)$ for $k=0,1,\ldots, q-1$, where $\overline{E}=\{(i,j):i<j,\{i,j\}\in E(\overline{G})\}$.
\item The $(i,j)$-row of $\verA(A)$ is $\vect_{E_o}(AX_{ij})$.
\end{enumerate}
\end{prop}
\begin{proof}
Let $n$ be the number of vertices of $G$.  For fixed $i,j$ with $1\leq i\leq j\leq n$, the $(i,j)$-row of $\verS(A)$ is the $(i,j)$-row of $\TSS(A)$ by definition, and the $(k,\ell)$-entry is 

\[\begin{aligned}
\be_i\trans(AK_{k\ell}+K_{\ell k}A)\be_j &=\be_i\trans(AE_{k\ell}-AE_{\ell k}+E_{\ell k}A-E_{k\ell}A)\be_j \\
 &= \be_i\trans(A\be_k\be_\ell\trans-A\be_\ell\be_k\trans+\be_\ell\be_k\trans A-\be_k\be_\ell\trans A)\be_j \\
 &= \be_i\trans A\be_k\be_\ell\trans\be_j-\be_i\trans A\be_\ell\be_k\trans\be_j+\be_i\trans \be_\ell\be_k\trans A\be_j-\be_i\trans \be_k\be_\ell\trans A\be_j \\
 &= \be_\ell\trans\be_j\be_i\trans A\be_k-\be_k\trans\be_j\be_i\trans A\be_\ell+\be_k\trans A\be_j\be_i\trans \be_\ell-\be_\ell\trans A\be_j\be_i\trans \be_k \\
 &= \be_k\trans A\be_i\be_j\trans\be_\ell-\be_k\trans\be_j\be_i\trans A\be_\ell+\be_k\trans A\be_j\be_i\trans \be_\ell-\be_k\trans\be_i\be_j\trans A\be_\ell \\
 &=\be_k\trans(AE_{ij}-E_{ji}A+AE_{ji}-E_{ij}A)\be_\ell \\
 &=\be_k\trans(AX_{ij}-X_{ij}A)\be_\ell, \\
\end{aligned}\]
which is the $(k,\ell)$-entry of $AX_{ij}-X_{ij}A$.  This deals with (1).  For (2), it follows directly from Definition \ref{TSmatrices} and Definition \ref{vermtxdef}.  Finally, for the $(i,j)$-row of $\verA(A)$, the $(k,\ell)$-entry is 
\[\begin{aligned}
\be_i\trans(AE_{k\ell}+E_{\ell k}A)\be_j 
 &= \be_i\trans A\be_k\be_\ell\trans\be_j+\be_i\trans\be_\ell\be_k\trans A\be_j \\
 &= \be_\ell\trans\be_j\be_i\trans A\be_k+\be_k\trans A\be_j\be_i\trans\be_\ell \\
 &= \be_k\trans A\be_i\be_j\trans\be_\ell+\be_k\trans A\be_j\be_i\trans\be_\ell \\
 &= \be_k\trans (AX_{ij})\be_\ell,
\end{aligned}\]
which is the $(k,\ell)$-entry of $AX_{ij}$.  This gives (3).
\end{proof}


\subsection*{Acknowledgments}  We thank the Banff International Research Station for providing a wonderful research environment where most of this research was done.  We thank the American Institute of Mathematics for bringing the authors together, both originally more than ten years ago, and recently, which allowed us to further advance this work. 

\end{document}